\documentclass{mcom-l}

\usepackage{graphicx}%
\usepackage{multirow}%
\usepackage{amsmath,amssymb,amsfonts}%
\usepackage{amsthm}%
\usepackage{mathrsfs}%
\usepackage[title]{appendix}%
\usepackage{xcolor}%
\usepackage{textcomp}%
\usepackage{manyfoot}%
\usepackage{booktabs}%
\usepackage{algorithm}%
\usepackage{algorithmicx}%
\usepackage{algpseudocode}%
\usepackage{listings}%

\usepackage{hyperref}
\usepackage[noabbrev, nameinlink]{cleveref}

\theoremstyle{thmstyleone}%
\newtheorem{theorem}{Theorem}
\newtheorem{proposition}[theorem]{Proposition}%
\newtheorem{lemma}[theorem]{Lemma}
\newtheorem{corollary}[theorem]{Corollary}

\theoremstyle{thmstyletwo}%
\newtheorem{example}{Example}%
\newtheorem*{remark*}{Remark}%

\theoremstyle{thmstylethree}%
\newtheorem{definition}{Definition}%
\newtheorem{notation}{Notation}%
\newif\ifmytrue

\usepackage[shortlabels, inline]{enumitem}
\usepackage{physics}
\usepackage{subcaption}
\usepackage{bm}
\usepackage{mathtools}
\newcommand*{\tensor}[1]{\mathbf{#1}}
\renewcommand*{\matrix}[1]{\mathbf{#1}}
\renewcommand*{\vector}[1]{\bm{#1}}
\newcommand{\hatcirc}{\mathbin{\hat{\circ}}}

\makeatletter
\newcommand\notsotiny{\@setfontsize\notsotiny\@vipt\@viipt}
\makeatother

\DeclareFontFamily{U}{mathx}{\hyphenchar\font45}
\DeclareFontShape{U}{mathx}{m}{n}{
      <5> <6> <7> <8> <9> <10>
      <10.95> <12> <14.4> <17.28> <20.74> <24.88>
      mathx10
      }{}
\DeclareSymbolFont{mathx}{U}{mathx}{m}{n}
\DeclareFontSubstitution{U}{mathx}{m}{n}
\DeclareMathSymbol{\bigtimes}{1}{mathx}{"91}

\usepackage{subcaption}

\newdimen\myshiftaligndimen
\makeatletter
\usepackage{tikz}
\usepackage{forest}
\forestset{
    declare dimen register={shift align},
    shift align'=0pt,
    default preamble={
        TeX={%
          \pgfinterruptpicture
          \myshiftaligndimen=\fontdimen2\font@name
          \global\myshiftaligndimen\myshiftaligndimen
          \endpgfinterruptpicture
        },
        /tikz/baseline={([yshift=-0.75\myshiftaligndimen]current bounding box.center)},
        for tree={%
            grow=north,
            content=,
            circle,
            fill,
            minimum size=2.5pt,
            inner sep=0.5pt,
            l=3pt,
            l sep=3pt,
            s sep=2pt,
            edge={-, thick},
        },
    },
    midtree/.style={
         /tikz/every node/.append style={font=\small},
        for tree={%
            label/.option=content,
            grow=north,
            content=,
            circle,
            fill,
            minimum size=3pt,
            inner sep=0pt,
            l=8pt,
            l sep=8pt,
            s sep=9pt,
            edge={-, very thick},
        },
    },
    bigtree/.style={
        for tree={%
            label/.option=content,
            grow=north,
            content=,
            circle,
            fill,
            minimum size=4pt,
            inner sep=0.5pt,
            l=18pt,
            l sep=15pt,
            s sep=15pt,
            edge={-, very thick},
        },
    },
    EL/.style 2 args={
        edge label={node[midway, font=\sffamily\scriptsize, #1]{#2}},
    },
}

\begin{document}

\title[Derivation of RK Order Conditions via Functional TTNs]{Derivation of Runge--Kutta Order Conditions via Functional Tree Tensor Networks}

\author{Junyuan He}
\address{SKLMS, Academy of Mathematics and Systems Science, Chinese Academy of Sciences; School of Mathematical Sciences, University of Chinese Academy of Sciences}
\curraddr{55 Zhongguancun East Rd., Beijing, 100190, PR China}
\email{hejunyuan@lsec.cc.ac.cn}

\author{Zhonghao Sun}
\address{SKLMS, Academy of Mathematics and Systems Science, Chinese Academy of Sciences; School of Mathematical Sciences, University of Chinese Academy of Sciences}
\curraddr{55 Zhongguancun East Rd., Beijing, 100190, PR China}
\email{sunzhonghao@amss.ac.cn}

\author{Jizu Huang}
\address{SKLMS, Academy of Mathematics and Systems Science, Chinese Academy of Sciences; School of Mathematical Sciences, University of Chinese Academy of Sciences}
\curraddr{55 Zhongguancun East Rd., Beijing, 100190, PR China}
\email{huangjz@lsec.cc.ac.cn}
\thanks{}

\subjclass[2010]{65L06, 46M05, 15A69}

\date{}

\dedicatory{}

\begin{abstract}
Tree tensor networks (TTNs) provide a compact and structured representation of high-dimensional data, making them valuable in various areas of computational mathematics and physics. In this paper, we present a rigorous mathematical framework for expressing high-order derivatives of functional TTNs, both with or without constraints. Our framework decomposes the total derivative of a given TTN into a summation of TTNs, each corresponding to the partial derivatives of the original TTN. Using this decomposition, we derive the Taylor expansion of vector-valued functions subject to ordinary differential equation constraints or algebraic constraints imposed by Runge--Kutta (RK) methods. As a concrete application, we employ this framework to construct order conditions for RK methods. Due to the intrinsic tensor properties of partial derivatives and the separable tensor structure in RK methods, the Taylor expansion of numerical solutions can be obtained in a manner analogous to that of exact solutions using tensor operators. This enables the order conditions of RK methods to be established by directly comparing the Taylor expansions of the exact and numerical solutions, eliminating the need for mathematical induction. For a given function $\vector{f}$, we derive sharper order conditions that go beyond the classical ones, enabling the identification of situations where a standard RK scheme of order {\it p} achieves unexpectedly higher convergence order for the particular function. These results establish new connections between tensor network theory and classical numerical methods, potentially opening new avenues for both analytical exploration and practical computation.
\end{abstract}

\maketitle


\mytruetrue

\section{Introduction}


Tensor-valued functions provide a natural generalization of scalar and vector functions, mapping between tensor spaces in both their domain and range. As fundamental elements in modern computational mathematics, they enable tractable computation in high-dimensional spaces through their inherent structural properties. Their ability of tensor functions to efficiently represent and manipulate complex multidimensional data has established them as indispensable mathematical tools with wide-ranging applications, such as machine learning \cite{ji2019survey,liu2019machine, liu2023tensor,stoudenmire2016supervised}, numerical analysis \cite{bachmayr2023low, grasedyck2013literature,hackbusch2012tensor}, and scientific computing \cite{hauschild2018efficient, khoromskij2012tensors,schroder2019tensor}. Furthermore, tensor functions have demonstrated remarkable efficacy in modeling complex physical phenomena where traditional methods face limitations, particularly in quantum many-body systems \cite{markov2008simulating,orus2014practical}, computational fluid dynamics \cite{gourianov2022exploiting,gourianov2024tensor}, and multidimensional signal processing \cite{acar2008unsupervised}.

To address the challenges of  high-dimensional tensor representation, various low rank tensor decomposition methods have been developed over recent decades, including CANDECOMP/PARAFAC (CP) \cite{kiers2000towards}, Tucker \cite{tucker1966some}, tensor-train (TT) \cite{oseledets2011tensor} and other tensor networks \cite{shi2006classical}. While originally developed for discrete settings, these decomposition techniques admit natural extensions to continuous function spaces \cite{khoromskij2015tensor}. The CP decomposition factorizes a tensor into a sum of rank-one components, providing a compact representation through separable terms. Tucker decomposition represents tensors via a core tensor and factor matrices, offering effective dimensionality reduction capabilities \cite{cichocki2016tensor}. 
TT decomposition parameterizes multilinear operations through sequential low rank matrix products \cite{oseledets2011tensor}, significantly reducing memory requirements and enabling efficient computations \cite{dolgov2012superfast}. Meanwhile, tensor networks, inspired by quantum many-body physics \cite{orus2019tensor,shi2006classical}, employ interconnected tensor cores in structured network topologies to capture complex, high-order interactions.
The remarkable efficiency of low-rank tensor representations has inspired extensive research into their integration with neural network architectures for deep learning applications \cite{wang2023tensor}. Notable examples include CP-CNN \cite{lebedev2015speeding}, Tucker-CNN \cite{phan2020stable}, TT-CNN \cite{novikov2015tensorizing}, TT-RNN \cite{yang2017tensor}, Conv-TT-LSTM \cite{su2020convolutional}, BTT-Transformer \cite{ma2019tensorized}, PMO-Transformer \cite{liu2021enabling}, TGNNS \cite{hua2022high}, etc. These tensor-based architectures leverage the inherent compressibility of tensor decompositions to achieve parameter efficiency while preserving model expressiveness.
This paradigm has demonstrated success across diverse applications, including neural network compression, multimodal information fusion, and quantum circuit simulation.


Low rank tensor functions have emerged as powerful tools across diverse applications, typically represented as sums of lower-order tensor functions. However, computing their derivatives, particularly for higher orders, presents significant computational challenges. This difficulty is particularly pronounced when dealing with constrained vector-valued functions, such as those governed by ordinary differential equation (ODE) constraints or algebraic constraints imposed by Runge--Kutta (RK) methods. The representation of high-order derivatives for vector-valued functions proves crucial in establishing order conditions for RK methods. The evolution of RK methods has been closely tied to the derivation of these conditions, which provide necessary and sufficient criteria for guaranteeing numerical accuracy. Carl Runge \cite{runge1895numerische} and Martin Kutta \cite{kutta1901beitrag} laid the groundwork for iterative integration processes, but it was the formalization of order conditions that enabled these methods to achieve reliable accuracy in solving ODEs. In the 1960s, Butcher introduced algebraic structures, such as Butcher trees, to systematically derive and analyze order conditions for both explicit and implicit RK methods \cite{butcher1963coefficients,butcher2016numerical}. These developments facilitated the development of high-order schemes, implicit schemes for stiff problems, and embedded RK pairs for adaptive time stepping, solidifying RK methods as indispensable tools for numerically solving ODEs.

Building upon Butcher's tree-based framework, researchers proposed several alternative approaches for deriving order conditions. The B-series formalism emerged as a powerful analytical tool, enabling systematic comparisons between the Taylor expansions of exact ODE solutions and their numerical approximations \cite{butcher1963coefficients,butcher2021b,chartier2010algebraic,hairer1987solving}. This foundational work has been extended and adapted for analyzing partitioned RK methods \cite{hairer1981order, murua1997order}, exponential integrators \cite{berland2005b}, stochastic RK methods \cite{burrage2000order}, etc. Lie-Butcher theory introduced another alternative perspective, constructing order conditions through vector field commutators \cite{munthe1995lie,munthe1998runge}. This approach has proven particularly valuable in structure-preserving geometric integration \cite{bruls2010use,munthe1999computations} and high-order symplectic algorithms \cite{munthe1999high}. Most recently, innovations in nonlinearly partitioned RK methods \cite{buvoli2024new,tran2024order} have motivated the search for simpler and more universal ways to establish order conditions.

In this work, we propose a novel framework based on tree tensor networks (TTNs) \cite{bachmayr2016tensor} to efficiently compute and represent high-order derivatives of constrained vector-valued functions. This framework enables the systematic derivation of order conditions for RK methods through the following key innovations:
\noindent\begin{enumerate}
\item {\bf Recursive derivative decomposition:} The $k^{\mathrm{th}}$-order total derivative of constrained vector-valued functions is decomposed into a sum of $k^{\mathrm{th}}$-order partial derivatives, each computed by differentiating the corresponding tensor cores of the $(k-1)^{\mathrm{th}}$-order partial derivative.
\item {\bf Diagrammatic representation and tensor calculus:} The resulting partial derivatives admit a natural representation as TTN diagrams, which share topological similarities to Butcher trees but differ fundamentally in their mathematical nature. Unlike Butcher’s abstract tree structures, our TTN framework operates on concrete partial derivatives with well-defined tensor operations.
\item {\bf Tree-derivative correspondence:} The framework establishes a fundamental connection between derivative operations and tree growth patterns. We present two distinct but equivalent perspectives: 

(i) Layer-wise growth from the root, where derivatives are taken as freely as possible before substituting constraints into the leaves. 

(ii) Leaf-wise expansion, where derivative is taken one at a time, immediately incorporating constraints into the new leaf.
\item {\bf Direct proof of order conditions:} Leveraging the tensor structure of TTNs, we achieve a constructive proof of order conditions without relying on mathematical induction. The terms in the Taylor expansion of numerical solutions decomposes naturally into: 
a) a contraction of method-dependent TTNs, determined by the RK tableau; and
b) a universal component identical to terms in the exact solution's expansion. This decomposition enables direct comparison of exact and numerical expansions, deriving order conditions through explicit tensor matching rather than inductive reasoning. 

\item {\bf Super convergence of an RK scheme for a given $\vector{f}$:} The classical order conditions $\gamma(T)\phi(T)=1,$ $\forall |T|\leq p$ are necessary for an RK scheme to achieve uniform order $p$ (i.e., convergence of order $p$ for all admissible $\vector{f}$). However, for a specific choice of $\vector{f}$, these conditions may not be necessary. By utilizing the TTN-based framework, we derive a refined set of order conditions tailored to each $\vector{f}$, as presented in \Cref{theorem:ordercondition-a}. These conditions enable us to identify situations where a standard RK scheme of order $p$ may exhibit a higher (superior) convergence order for that particular function.     
\end{enumerate}
A detailed comparison between our TTN-based method and Butcher's tree-based method is provided in \autoref{subsec:Comparison of our framework and Butcher's method}.

The remainder of the article is organized as follows. Section 2 introduces (functional) tensor networks, including fundamental tensor notations and operations, and presents our framework for computing derivatives of TTNs, both with and without constraints. Utilizing our TTNs derivative framework, the Taylor expansion of vector-valued functions subject to ODE constraints and satisfying constraints imposed by RK method, are explored in Section 3 and Section 4, respectively.  Section 5 applies the proposed framework to derive order conditions for RK methods. Section 6 concludes the article with a summary of key findings.

\section{Functional tree tensor networks and their derivatives}

In this section, we present a framework for computing the derivatives of functional TTNs. We begin by a brief review of (functional) TTNs, and then propose a framework to compute its derivatives either without or with constraints. The functional TTNs and its derivatives are utilized to construct order conditions of RK methods in \Cref{sec:order-condition-RK}.

\subsection{Preliminary operations on tensors}

We primarily introduce Kronecker products, tensor contractions, and symmetry of tensors, which are used throughout this paper. For additional tensor operations, we refer the reader to \cite{lee2018fundamental}.

\subsubsection{Kronecker products}

For $\matrix{A} = (a_{ij}) \in \mathbb{R}^{m \times n}$ and $\matrix{B} = (b_{ij}) \in \mathbb{R}^{p \times q}$, the Kronecker product of $\matrix{A}$ and $\matrix{B}$ results a matrix $\matrix{C} := \matrix{A} \otimes \matrix{B} \in \mathbb{R}^{mp \times nq}$, which can be written as
\begin{equation*}
    \matrix{C} = \begin{bmatrix}
        a_{11} \matrix{B} & \cdots & a_{1n} \matrix{B} \\
        \vdots & & \vdots \\
        a_{m1} \matrix{B} & \cdots & a_{mn} \matrix{B}
    \end{bmatrix}.
\end{equation*}
In order to extend the Kronecker product to higher order tensors, we first introduce some necessary notations. The $(i_1,\, \dots,\, i_N)^\mathrm{th}$ element of a tensor $\tensor{A}$ with size $I_1 \times \cdots \times I_N$ is denoted as $\tensor{A}[i_1,\, \dots,\, i_N]$. The multi-index  $\overline{i_1 i_2 \dots i_N}$ is defined as $i_N + (i_{N-1} - 1) I_N + (i_{N-2} - 1) I_N I_{N-1} + \cdots + (i_1 - 1) I_N I_{N-1} \cdots I_2$, where $i_n = 1,2,\, \dots,\, I_n$ for each $n = 1,2,\,\dots,\, N$. Then, the Kronecker product of tensors $\tensor{A} \in \mathbb{R}^{I_1 \times I_2 \times \cdots \times I_N}$ and $\tensor{B} \in \mathbb{R}^{J_1 \times J_2 \times \cdots \times J_N}$ is defined as $\matrix{C}=\matrix{A} \otimes \matrix{B}$ with the $(\overline{i_1 j_1},\, \dots,\, \overline{i_N j_N})^{\mathrm{th}}$ element given by
\begin{equation*}
    \tensor{C}[\overline{i_1 j_1},\, \dots,\, \overline{i_N j_N}] := (\tensor{A} \otimes \tensor{B})[\overline{i_1 j_1},\, \dots,\, \overline{i_N j_N}] = \tensor{A}[i_1,\, \dots,\, i_N] \tensor{B}[j_1,\, \dots,\, j_N].
\end{equation*}

\subsubsection{Contraction}

Contraction, or contracted product on tensors, is a natural extension of the matrix product. Unless otherwise specified, the contracted product usually refers to the product of the last mode of the first tensor and the first mode of the second tensor. This is mathematically expressed by the following definition.

\begin{definition} \label{def:contraction}
    Assume two tensors $\tensor{A} \in \mathbb{R}^{I_1 \times I_2 \times \cdots \times I_M}$, $\tensor{B} \in \mathbb{R}^{J_1 \times J_2 \times \cdots \times J_N}$ with $I_M = J_1$, and $M,\,N\geq 1$. The contracted product (or more precisely $(M,1)$-contracted product) is defined by
    \begin{equation*}
        \tensor{C} := \tensor{A} \times^1 \tensor{B} \in \mathbb{R}^{I_1 \times \cdots \times I_{M-1} \times J_2 \times \cdots \times J_N}
    \end{equation*}
    with elements
    \begin{equation*}
        \tensor{C}[i_1,\, \dots,\, i_{M-1},\, j_2,\, \dots,\, j_{N}] = \sum_{i_M = 1}^{I_M} \tensor{A}[i_1,\, \dots,\, i_{M-1},\, i_M]\, \tensor{B}[i_M,\, j_2,\, \dots,\, j_N].
    \end{equation*}
\end{definition}

\noindent The matrix product is a special example of contracted product on tensors. For simplicity, $\tensor{A} \tensor{B}$ will be written instead of $\tensor{A} \times^1 \tensor{B}$ in the rest of this paper.

\begin{notation} \label{notation:contraction}
Denote the contraction of the last mode of tensor $\tensor{A}$ and the first mode of tensor $\tensor{B}$ as $\tensor{A} \times^1 \tensor{B}:=\tensor{A} \tensor{B}$. In general, the contraction is always computed from left to right, i.e., $\tensor{A}\tensor{B}\tensor{C}:=(\tensor{A}\tensor{B})\tensor{C}$.
\end{notation}

A useful property, known as the mixed product property, is presented in \cite{pollock2013kronecker, ragnarsson2012structured}. We state this property as the following lemma. 

\begin{lemma}\label{lemma:mixed-product-property}
    If $\tensor{A}$, $\tensor{B}$, $\tensor{C}$, and $\tensor{D}$ are tensors which can form contracted products $\tensor{A} \tensor{C}$ and $\tensor{B} \tensor{D}$, then 
    \begin{equation*}
        (\tensor{A} \otimes \tensor{B}) (\tensor{C} \otimes \tensor{D}) = (\tensor{A} \tensor{C}) \otimes (\tensor{B} \tensor{D}).
    \end{equation*}
\end{lemma}
This property transforms the contraction of Kronecker product into the Kronecker product of contractions, which is useful in computing decomposition and contraction of tensors in the remainder of our paper.

\subsubsection{Symmetry of tensors}
A tensor exhibits symmetry when its elements remain invariant under index permutations. A tensor is called {\it cubical} if all its modes have the same dimension, i.e., $\tensor{A}\in\mathbb{R}^{I\times I \times \cdots \times I}$. A cubical tensor $\tensor{A}$ of order $d$ is said to be symmetric if its elements remain unchanged under any permutation of indices \cite{comon2008symmetric}. This property is formally defined as follows. 
\begin{definition}
    Suppose $\tensor{A} \in \mathbb{R}^{I \times I \times \cdots \times I}$ is a $d^\mathrm{th}$ order tensor. $\tensor{A}$ is called symmetric if, for any permutation $\pi$ of $\{ 1, \, \dots, \, d \}$, 
    \begin{equation*}
        A[i_1,\, i_2,\, \dots,\, i_d] = A[i_{\pi(1)},\, i_{\pi(2)},\, \dots,\, i_{\pi(d)}].
    \end{equation*}
\end{definition}
More generally, tensors can exhibit \emph{partial symmetry}, meaning they remain invariant under permutations of a specific subset of indices. For example, a tensor of order $d$ can be symmetric in the last $(d-1)^{\mathrm{th}}$ indices but not in all indices.

High-order derivatives of a multivariate function $f: \mathbb{R}^n \to \mathbb{R}$ can be naturally represented as cubical tensors. Specifically, the $d^\mathrm{th}$ order derivative of $f \in C^{d}(\mathbb{R}^n)$ at a point $\vector{x}:=(x_1,\,\ldots,\,x_n)^{\mathrm T}$ is a $d^\mathrm{th}$ order cubical tensor, denoted as $f^{(d)}(\vector{x}) := {\mathcal D}^{d} f (\vector{x})\in\mathbb{R}^{n\times n\times\cdots \times n}$, whose components are given by
\begin{equation*}
     f^{(d)}[i_1,\,i_2,\,\dots,\,i_d](\vector{x}) = \frac{\partial^d f(\vector{x})}{\partial x_{i_1} \partial x_{i_2} \cdots \partial x_{i_d}}, \quad 1 \leq i_1,\,i_2,\,\dots,\,i_d \leq n.
\end{equation*}
By the commutativity of mixed partial derivatives, the tensor $f^{(d)}(\vector{x})$ is symmetric. 
Throughout this article, we frequently consider high-order derivatives of a vector-valued multivariate function $\vector{f}: \mathbb{R}^n \to \mathbb{R}^m$. For $\vector{f}$ with elements being $C^{d}(\mathbb{R}^n)$ smoothness, its $d^\mathrm{th}$ order derivative at $\vector{x}$ is a $(d+1)^{\mathrm{th}}$ order tensor $\vector{f}^{(d)}(\vector{x})\in\mathbb{R}^{m\times n\times\cdots \times n}$ with elements given by
\begin{equation*}
    \vector{f}^{(d)}[i,i_1,\,i_2,\,\dots,\,i_d](\vector{x}) = \frac{\partial^d \vector{f}[i](\vector{x})}{\partial x_{i_1} \partial x_{i_2} \cdots \partial x_{i_d}}, \quad 1 \leq i \leq m,\; 1 \leq i_1,\,i_2,\,\dots,\,i_d \leq n,
\end{equation*}
which exhibits partial symmetry in the last $d$ indices.

For tensors with symmetry or partial symmetry, contractions over symmetric indices commute. For example, for a symmetric matrix $\matrix{H} \in \mathbb{R}^{n\times n}$, it holds that $\vector{x}^\top \matrix{H} \vector{y} = \vector{y}^\top \matrix{H} \vector{x}$ for any two vectors $\vector{x}, \vector{y} \in \mathbb{R}^n$. More generally, according to \Cref{notation:contraction}, for vectors $\vector{v}_1,\, \dots, \,\vector{v}_d $ $\in \mathbb{R}^n$ and any permutation $\pi$ of $\{1,\,\dots,\, d\}$, we have 
\begin{equation}
\label{symmetryproperty}
    \vector{f}^{(d)}(\vector{x})\times^1 \vector{v}_1 \times^1\cdots\times^1 \vector{v}_d:=\vector{f}^{(d)}(\vector{x}) \vector{v}_1 \cdots \vector{v}_d = \vector{f}^{(d)}(\vector{x}) \vector{v}_{\pi(1)} \cdots \vector{v}_{\pi(d)},
\end{equation}
where $\vector{f}^{(d)}(\vector{x}) \vector{v}_1 \cdots \vector{v}_d \in\mathbb{R}^m$.
The property \eqref{symmetryproperty}, which allows us to contract the tensor $ \vector{f}^{(d)}(\vector{x})$ with vectors in any order, will be utilized in \Cref{sec:3}. Occasionally, we simplify the notation for $\vector{f}^{(d)}(\vector{x}) \vector{v}_1\cdots \vector{v}_d$ as follows:
\begin{equation}
\label{symmetryvectorpower}
    \vector{f}^{(d)}(\vector{x})\times^1 \vector{v}_1 \times^1\cdots\times^1 \vector{v}_d:=\vector{f}^{(d)}(\vector{x}) \vector{v}_1 \cdots \vector{v}_d: = \vector{f}^{(d)}(\vector{x}) \prod_{i=1}^d\vector{v}_i.
\end{equation}

\subsection{Tree tensor networks}

TTNs were first used by physicists. They were originally introduced in \cite{wang2003multilayer} to construct the multilayer multiconfiguration time-dependent Hartree theory of chemical physics. However, it wasn't until \cite{shi2006classical} the term ``tree tensor network'' was formally coined. Subsequently, it gained significant popularity in quantum systems \cite{murg2010simulating, nakatani2013efficient}. In \cite{bachmayr2023low, hackbusch2012tensor}, the low-rank approximations and the corresponding decomposition algorithms of binary trees, known as hierarchical tensors, were proposed. TTNs are a powerful and versatile tool for representing the decomposition of multivariate functions (tensors) into nested summations over contraction indices. We will present the basic concepts of TTNs following \cite{bachmayr2016tensor} and refer readers to \cite{bachmayr2016tensor,ceruti2021time,falco2015geometric,falco2021tree} for additional properties of TTNs.

Denote $i_\mu = 1,\,2,\,\ldots,\, n_\mu$ as the physical indices and $k_\nu=1,\,2,\,\ldots,\, r_\nu$ as the contraction indices. For a $d^\mathrm{th}$ order tensor $\tensor{U}$, we can define a multilinear parametrization of $\tensor{U}$ that separates the physical indices as the following form \cite{bachmayr2016tensor}:
\begin{equation}\label{eqn:multilinear-parameterization}
    \tensor{U}[i_1,\, \dots,\, i_d] = \sum_{k_1 = 1}^{r_1} \cdots \sum_{k_E = 1}^{r_E} \prod_{\alpha = 1}^{V} \tensor{C}_{\alpha}[i_1,\,\dots,\,i_d,\, k_1,\, \dots,\, k_E],
\end{equation}
where each component $\tensor{C}_\alpha$ potentially depends on all physical indices $i_1,\, \dots,\, i_d$ and contraction indices $k_1,\, \dots,\, k_E$. In \eqref{eqn:multilinear-parameterization}, $V$ is called the number of components. In the case of low-rank, $\tensor{C}_\alpha$ usually does not depend on all indices. If $\tensor{C}_\alpha$ does or does not depend on a certain index $i_\mu$ or $k_\nu$, then the index is called an \emph{active} or \emph{inactive index}.

\begin{definition}[{\cite[Definition 2.2]{bachmayr2016tensor}}] A \emph{tensor network} is defined as a particular type of multilinear parameterization where: 
\begin{enumerate}[(i), nosep]
    \item Each physical index $i_\mu$ is active in exactly one component $\tensor{C}_\alpha$;
    \item Each contraction index $k_\nu$ is active in precisely two components $\tensor{C}_{\alpha_1}$ and $\tensor{C}_{\alpha_2}$.
\end{enumerate}
\end{definition}

For a clearer description of contraction, a graph is introduced to present a tensor network \cite{bachmayr2016tensor}. For example, the tensor network $\tensor{U}$ defined in \eqref{eqn:multilinear-parameterization} is represented by a graph with vertices $\alpha=1,\,\ldots,\, V$, corresponding to components $\tensor{C}_\alpha$. These vertices are connected by edges $\nu=1,\,\ldots,\, E$. Each edge represents a summation over the corresponding contraction variable $k_\nu$. If a physical index $i_\mu$, $\mu=1,\,\ldots,\,d$ is active in component $\tensor{C}_\alpha$, an additional open edge is connected to the corresponding vertex. Thus, the number of open edges in the graph determines the order of tensor $\tensor{U}$. This kind of graphical representation is called tensor network diagrams \cite{cichocki2016tensor, holtz2012alternating, lee2018fundamental}.  Similar diagrammatic representations, such as Feynman and Goldstone diagrams \cite{goldstone1957derivation, mattuck1992guide}, are also commonly used in quantum physics to track summations. In \Cref{fig:tree-tensor-network-diagram}, some simple examples of tensor network diagrams are displayed to represent the contraction of tensors. 

\begin{figure}[htb]
    \centering
\ifmytrue
    $\vector{x}=\begin{forest}
    midtree
       [,fill=none  [,label=right:$\vector{x}$
        ] ]
    \end{forest}$
    \quad
    $\vector{x}\cdot \vector{x}=\begin{forest}
    midtree
       [,label=right:$\vector{x}$  [,label=right:$\vector{x}$
        ] ]
    \end{forest}$
    \quad
    $\matrix{A}\vector{x}=
    \begin{forest}
    midtree
       [,fill=none  [,label=right:$\matrix{A}$
            [,label=right:$\vector{x}$]
        ] ]
    \end{forest}$
    \quad
    $\vector{y}^\top \matrix{A}\vector{x} =\begin{forest}
    midtree
       [,label=right:$\vector{y}$ [,label=right:$\matrix{A}$
            [,label=right:$\vector{x}$]
        ] ]
    \end{forest}$
    \quad 
    $\tensor{B}=
        \begin{forest}
    midtree
       [,fill=none  [,label=right:$\tensor{B}$
            [,fill=none]
            [,fill=none]
        ] ]
    \end{forest}  $ 
    \quad 
    $ \tensor{B}\vector{x}= 
        \begin{forest}
    midtree
       [,fill=none  [,label=right:$\tensor{B}$
            [,fill=none]
            [,label=above:$\vector{x}$]
        ] ]
    \end{forest} \in\mathbb{R}^{n_2\times n_3}$  
    \quad
    $(\tensor{B}\vector{x})(\matrix{A}\vector{x}) =  
        \begin{forest}
    midtree
       [,fill=none  [,label=right:$\tensor{B}$
            [,label=right:$\matrix{A}$ [,label=above:$\vector{x}$]]
            [,label=above:$\vector{x}$]
        ] ]
    \end{forest}\in\mathbb{R}^{ n_3}$  
    \quad 
    $ \vector{z}^\top (\tensor{B} \vector{x})(\matrix{A} \vector{x})=
        \begin{forest}
    midtree
       [,label=right:$\vector{z}$  [,label=right:$\tensor{B}$
            [,label=right:$\matrix{A}$ [,label=above:$\vector{x}$]]
            [,label=above:$\vector{x}$]
        ] ]
    \end{forest}$ 
\fi
    \caption{The tensor network diagram. Here $\vector{x}\in\mathbb{R}^{n_1}$, $\vector{y}\in\mathbb{R}^{n_2}$, $\vector{z}\in\mathbb{R}^{n_3}$, $\matrix{A} \in \mathbb{R}^{n_2\times n_1}$, and $\tensor{B} \in \mathbb{R}^{n_3\times n_2\times n_1}$, respectively. The notation of contracted products follow \Cref{notation:contraction}, i.e., $\tensor{B} \vector{x}:=\tensor{B}\times^1 \vector{x}$. }
    \label{fig:tree-tensor-network-diagram}
\end{figure}

A tensor network is referred to as a \emph{tree tensor network} (TTN) if its graph structure is a tree, meaning it contains no loops or cycles \cite{bachmayr2016tensor}. The widely used tensor networks, i.e., the Tucker, hierarchical Tucker, and TT formats all are TTNs. By assigning a root to the tree, a TTN can be viewed as a multilevel Tucker tensor \cite{ceruti2021time} or tree-based Tucker tensor \cite{falco2021tree}. Without loss of generality, we always assign the component $\tensor{C}_1$ as the root for the TTN defined by \eqref{eqn:multilinear-parameterization}. For notational convenience, we remove all inactive indices from the indices of components. Denote the set of active contraction indices and active physical indices in component $\tensor{C}_\alpha$ as $\mathbb{E}_\alpha$ and $\mathbb{E}_\alpha^o$, respectively. Let us denote $\mathbb{E}=\cup_{\alpha}\mathbb{E}_\alpha$ and $\mathbb{E}^o=\cup_{\alpha}\mathbb{E}^o_\alpha$. Using these notations, TTNs can be written in a more compact way. A TTN of a scalar $u$ dose not include any open active index and can be written as
\begin{equation*}
    u = \sum_{\mathbb{E}} \prod_{\alpha=1}^{V} \mathbf{C}_{\alpha}[\mathbb{E}_\alpha],
\end{equation*}
where $\mathbb{E}_\alpha \subset \mathbb{E} =\{k_1,\,\ldots,\,k_E\}$, $\alpha=1,\,\dots,\,d$, are index sets of the contraction edges connected with the vertex $\alpha$. Here, $\sum\limits_{\mathbb{E}}$ is defined as $\sum\limits_{\mathbb{E}}:=\sum\limits_{k_1=1}^{r_1}\cdots \sum\limits_{k_E=1}^{r_E}$ with $E=|\mathbb{E}|$. For a $d^\mathrm{th}$ order tensor $\tensor{U}$, the TTN can be denoted as: 
\begin{equation}
    \tensor{U}[\mathbb{E}^o] = \sum_{\mathbb{E}} \prod_{\alpha=1}^{V} \tensor{C}_{\alpha}[\mathbb{E}_{\alpha}^o,\, \mathbb{E}_\alpha],
\end{equation}
where $\mathbb{E}^o:=\{i_1,\,\ldots,\,i_d\}$ with $|\mathbb{E}^o|=d$. 

For instance, assume that $\tensor{U}\in \mathbb{R}^{n_1\times n_2}$ is a $2^\mathrm{nd}$-order tensor defined as follows:
\begin{align*}
    \tensor{U}[i_1,i_2] &= \sum_{k_1,\dots,k_5} \tensor{C}_1[i_1,k_1,k_2] \tensor{C}_2[i_2,k_1,k_3] \tensor{C}_3[k_3,k_4,k_5] \tensor{C}_4[k_4] \tensor{C}_5[k_5] \tensor{C}_6[k_2] \\
    &:= \sum\limits_{\mathbb{E}} \prod\limits_{\alpha=1}^{6} \tensor{C}_\alpha[\mathbb{E}_\alpha^o,\,\mathbb{E}_\alpha] = \tensor{U}[\mathbb{E}^o],
\end{align*}
where $\tensor{C}_1 \in \mathbb{R}^{n_1 \times r_1 \times r_2}$, $\tensor{C}_2 \in \mathbb{R}^{n_2 \times r_1 \times r_3}$, $\tensor{C}_3 \in \mathbb{R}^{r_3 \times r_4 \times r_5}$, $\tensor{C}_4 \in \mathbb{R}^{r_4}$, $\tensor{C}_5 \in \mathbb{R}^{r_5}$, and $\tensor{C}_6 \in \mathbb{R}^{r_2}$, respectively. The configuration of $\tensor{U}$ is given in \Cref{fig:example-tree-tensor-a} and the TTN diagram is displayed in \Cref{fig:example-tree-tensor-b}, respectively. In this work, for a TTN diagram, the edges connected to each vertex are arranged from left to right in increasing order of contraction indices or physical indices.
For clarity and conciseness, as shown in \Cref{fig:example-tree-tensor-c}, we omit the contraction sizes and physical sizes in the TTN diagrams when no ambiguity arises.

\begin{figure}[htb]
    \centering
    \subcaptionbox{Configuration \label{fig:example-tree-tensor-a}}[0.253\linewidth]{
\ifmytrue
    \begin{forest}
        bigtree
        [,fill=none 
        [,label=left:$\tensor{C}_1$ , EL={left=-2pt}{$i_1$}
            [,label=right:$\tensor{C}_6$, EL={right=-2pt}{$k_2$}]
            [,label=left:$\tensor{C}_2$, EL={left=-2pt}{$k_1$}
                [,fill=none, EL={right=-2pt}{$i_2$}]
                [,label=left:$\tensor{C}_3$, EL={left=-2pt}{$k_3$}
                    [,label=above:$\tensor{C}_5$, EL={right=-2pt}{$k_5$}] 
                    [,label=above:$\tensor{C}_4$, EL={left=-2pt}{$k_4$}]
                ]
            ]
        ]
        ]
    \end{forest}
\fi    
    }
    \subcaptionbox{TTN diagram \label{fig:example-tree-tensor-b}}[0.253\linewidth]{
\ifmytrue
    \begin{forest}
        bigtree
        [,fill=none 
        [,label=left:$\tensor{C}_1$ , EL={left=-2pt}{$n_1$}
            [,label=right:$\tensor{C}_6$, EL={right=-2pt}{$r_2$}]
            [,label=left:$\tensor{C}_2$, EL={left=-2pt}{$r_1$}
                [,fill=none, EL={right=-2pt}{$n_2$}]
                [,label=left:$\tensor{C}_3$, EL={left=-2pt}{$r_3$}
                    [,label=above:$\tensor{C}_5$, EL={right=-2pt}{$r_5$}] 
                    [,label=above:$\tensor{C}_4$, EL={left=-2pt}{$r_4$}]
                ]
            ]
        ]
        ]
    \end{forest}
\fi
    }
    \subcaptionbox{Simplified TTN diagram \label{fig:example-tree-tensor-c}}[0.3\linewidth]{
\ifmytrue
    \begin{forest}
        bigtree
        [,fill=none 
        [,label=left:$\tensor{C}_1$ , EL={left=-2pt}{}
            [,label=right:$\tensor{C}_6$, EL={right=-2pt}{}]
            [,label=left:$\tensor{C}_2$, EL={left=-2pt}{}
                [,fill=none, EL={right=-2pt}{}]
                [,label=left:$\tensor{C}_3$, EL={left=-2pt}{}
                    [,label=above:$\tensor{C}_5$, EL={right=-2pt}{}] 
                    [,label=above:$\tensor{C}_4$, EL={left=-2pt}{}]
                ]
            ]
        ]
        ]
    \end{forest}
\fi
    }
    \caption{An example of a $2^\textrm{nd}$-order tree tensor network. In the configuration diagram, edge labels are physical and contraction indices. In the TTN diagram, edge labels are dimensions of the corresponding mode.}
    \label{fig:example-tree-tensor}
\end{figure}

\subsection{The functional tree tensor networks and their derivatives}
\label{TTNderivatives}

Similar to TT or other tensor representations \cite{dolgov2021functional, gorodetsky2019continuous}, the TTN also possesses a continuous analogue. By replacing $\tensor{C}_\alpha$ with tensor-valued functions $\tensor{f}_\alpha $, we extend a TTN to a \emph{functional tree tensor network}. Typically, for $\vector{x}_\mu\in\mathbb{R}^{m_\mu}$ with $\mu = 1,\,\ldots,\,d$, we denote elements of the functional TTN $\tensor{F}(\vector{x}_1,\, \dots,\, \vector{x}_d)$ as
\begin{equation}\label{eqn:functional-TTN}
    \tensor{F}[\mathbb{E}^o](\vector{x}_1,\, \dots,\, \vector{x}_d) = \sum_{\mathbb{E}} \prod_{\alpha=1}^{V} \tensor{f}_{\alpha}[\mathbb{E}_{\alpha}^o,\, \mathbb{E}_{\alpha}](\Tilde{\vector{x}}_\alpha),
\end{equation}
where $\Tilde{\vector{x}}_\alpha \in \mathbb{R}^{\Tilde{m}_\alpha}$ is a vector with elements selected from $(\vector{x}_1,\,\ldots,\,\vector{x}_d)$. In the rest of this paper, we sometimes simplify TTN $\tensor{F}(\vector{x}_1, \, \cdots, \, \vector{x}_d)$ as $\tensor{F}$. 
For ease of explanation, assume $V=d$, $ \tilde{\vector{x}}_\alpha = \vector{x}_\alpha$, and introduce the definition of derivatives for the functional TTN $\tensor{F}$. This definition can be extended to more complex cases, but we omit the details due to page limitations.

Due to the separable representation of $\tensor{F}$, the partial derivative of $\mathbf{F}$ with respect to $\vector{x}_{\alpha_0}$ is equivalent to differentiate on the $\alpha_0^\mathrm{th}$ component function $\tensor{f}_{\alpha_0}$, and then adding an additional open mode to the resulting component $\tensor{f}'_{\alpha_0}$. We denote the resulting partial derivative by $\mathfrak{D}_{\alpha_0}\tensor{F}$. This operation corresponds to adding a physical index to $\mathbb{E}_{\alpha_0}^o$. The new set of physical indices of component $\tensor{f}'_{\alpha_0}$ and $\mathfrak{D}_{\alpha_0}\tensor{F}$ is denoted as $\Tilde{\mathbb{E}}_{\alpha_0}^o = \mathbb{E}_{\alpha_0}^o \cup \{ i_{|\mathbb{E}^o|+1} \}$ and $\Tilde{\mathbb{E}}^o = \mathbb{E}^o \cup \{ i_{|\mathbb{E}^o|+1} \}$, respectively. With these notations, we can express elements of partial derivative $\mathfrak{D}_{\alpha_0}\tensor{F}$ as:
\begin{equation}\label{eqn:tensor-derivative}
    (\mathfrak{D}_{\alpha_0} \tensor{F})[\Tilde{\mathbb{E}}^o](\vector{x}_1,\, \dots,\, \vector{x}_d) = \sum_{\mathbb{E}} \left[ \tensor{f}'_{\alpha_0}[\Tilde{\mathbb{E}}_{\alpha_0}^o,\, \mathbb{E}_{\alpha_0}](\vector{x}_{\alpha_0}) \prod_{ \substack{\alpha=1 \\ \alpha \neq \alpha_0}}^{d} \tensor{f}_{\alpha}[\mathbb{E}_{\alpha}^o,\, \mathbb{E}_\alpha](\vector{x}_\alpha) \right].
\end{equation}
It is important to note that we can obtain the tensor $\mathfrak{D}_{\alpha_0}\tensor{F}$ by only adjusting the $\alpha_0^\mathrm{th}$ component of $\tensor{F}$ as defined in \eqref{eqn:tensor-derivative}. 

The partial derivative of a TTN can be easily understood and implemented using the TTN diagram. Let us take a second-order tensor-valued function $\tensor{F}$ as an example. The TTN diagram of $\tensor{F}$ is displayed in \Cref{fig:example-partial-derivative-TTN-a}. By taking the derivative of the variable at the $6^\mathrm{th}$ core, the resulting $\mathfrak{D}_6 \tensor{F}$ is a third-order tensor-valued function, represented by the TTN diagram in \Cref{fig:example-partial-derivative-TTN-b}. The dimension of the new physical mode is the same as the dimension of the variable we are differentiating on.

\begin{figure}[htb]
    \centering
    \subcaptionbox{$\tensor{F}$ \label{fig:example-partial-derivative-TTN-a}}[0.3\linewidth]{
\ifmytrue
        \begin{forest}
            bigtree
            [,fill=none 
            [,label=left:$\tensor{f}_1$ , EL={left=-2pt}{$n_1$}
                [,label=right:$\tensor{f}_6$, EL={right=-2pt}{$r_2$}]
                [,label=left:$\tensor{f}_2$, EL={left=-2pt}{$r_1$}
                    [,fill=none, EL={right=-2pt}{$n_2$}]
                    [,label=left:$\tensor{f}_3$, EL={left=-2pt}{$r_3$}
                        [,label=above:$\tensor{f}_5$, EL={right=-2pt}{$r_5$}] 
                        [,label=above:$\tensor{f}_4$, EL={left=-2pt}{$r_4$}]
                    ]
                ]
            ]
            ]
        \end{forest}
\fi
    }
    \subcaptionbox{$\mathfrak{D}_6 \tensor{F}$ \label{fig:example-partial-derivative-TTN-b}}[0.3\linewidth]{
\ifmytrue
        \begin{forest}
            bigtree
            [,fill=none 
            [,label=left:$\tensor{f}_1$ , EL={left=-2pt}{$n_1$}
                [,label=right:$\tensor{f}'_6$, EL={right=-2pt}{$r_2$}
                    [,fill=none, EL={right=-2pt}{$m_6$}]
                ]
                [,label=left:$\tensor{f}_2$, EL={left=-2pt}{$r_1$}
                    [,fill=none, EL={right=-2pt}{$n_2$}]
                    [,label=left:$\tensor{f}_3$, EL={left=-2pt}{$r_3$}
                        [,label=above:$\tensor{f}_5$, EL={right=-2pt}{$r_5$}] 
                        [,label=above:$\tensor{f}_4$, EL={left=-2pt}{$r_4$}]
                    ]
                ]
            ]
            ]
        \end{forest}
\fi
    }
    \subcaptionbox{$(\mathfrak{D}_6 \tensor{F}) \vector{y}'$ \label{fig:example-partial-derivative-TTN-c}}[0.3\linewidth]{
\ifmytrue
        \begin{forest}
            bigtree
            [,fill=none 
            [,label=left:$\tensor{f}_1$ , EL={left=-2pt}{$n_1$}
                [,label=right:$\tensor{f}'_6$, EL={right=-2pt}{$r_2$}
                    [,label=above:{$\vector{y}'$}, EL={right=-2pt}{$m_6$} ]
                ]
                [,label=left:$\tensor{f}_2$, EL={left=-2pt}{$r_1$}
                    [,fill=none, EL={right=-2pt}{$n_2$}]
                    [,label=left:$\tensor{f}_3$, EL={left=-2pt}{$r_3$}
                        [,label=above:$\tensor{f}_5$, EL={right=-2pt}{$r_5$}] 
                        [,label=above:$\tensor{f}_4$, EL={left=-2pt}{$r_4$}]
                    ]
                ]
            ]
            ]
        \end{forest}
\fi
    }
    \caption{An example of the partial derivative of functional TTNs.}
    \label{fig:derivative-tree-tensor}
\end{figure}

\subsubsection{The differential of functional tree tensor networks}

With the definition of the partial derivative of functional TTNs, we can extend the concept of differential from scalar-valued or vector-valued functions to functional TTNs. For a functional TTN $\tensor{F}$ as defined in \eqref{eqn:functional-TTN}, the differential of $\tensor{F}$ at $(\vector{x}_1, \,\ldots\,\vector{x}_d)$, denoted by $\dd \tensor{F}$, can be defined as follows
\begin{equation}\label{eqn:differential-d-var}
    \dd \tensor{F} := \dd \tensor{F}(\vector{x}_1, \,\ldots\,\vector{x}_d) = \sum_{\alpha=1}^d \mathfrak{D}_\alpha \tensor{F} \dd \vector{x}_\alpha.
\end{equation}
For each $\alpha$, $\mathfrak{D}_\alpha \tensor{F}$ includes a new physical mode with dimension $m_\alpha$ and $\dd \vector{x}_\alpha $ is a differential that is a vector of dimension $m_\alpha$. The contraction of the TTN $\mathfrak{D}_\alpha \tensor{F}$ and $\dd \vector{x}_\alpha $ results in a TTN, with the same order as $\tensor{F}$. Therefore, summing over the $\alpha$'s, the differential $\dd \tensor{F}$ defined by \eqref{eqn:differential-d-var} is also a TTN of the same order as $\tensor{F}$. Next, using the definition of differential, we establish the relationship between the total derivatives and partial derivatives of functional TTNs. By assuming $\vector{x}_\alpha = \vector{y}$, $\forall \alpha = 1,\dots,d$, we can rewrite the elements of $|\mathbb{E}^o|^{\mathrm{th}}$ order tensor $\tensor{F}$ given in \eqref{eqn:functional-TTN} as
\begin{equation}\label{eqn:original-tensor}
    \tensor{F}[\mathbb{E}^o](\bm{y}) = \sum_{\mathbb{E}} \prod_{\alpha=1}^{d} \tensor{f}_\alpha[\mathbb{E}_{\alpha}^o, \, \mathbb{E}_\alpha](\vector{y}).
\end{equation}
The total derivative of $\tensor{F}$ with respect to $\vector{y}$ is defined as 
\begin{equation}\label{eqn:differential-1-var}
    \tensor{F}'(\vector{y}) := \frac{\dd \tensor{F}}{\dd \vector{y}} = \sum_{\alpha=1}^{d} (\mathfrak{D}_{\alpha} \tensor{F})(\vector{y}),
\end{equation}
which is a $(|\mathbb{E}^o|+1)^{\mathrm{th}}$ order tensor with a new physical mode of dimension $\mathbb{R}^{m_1}$ as compared to $\tensor{F}$. The $\alpha^\mathrm{th}$ term $(\mathfrak{D}_{\alpha} \tensor{F})(\vector{y})$ in the summation is a $(|\mathbb{E}^o| + 1)^\text{th}$ order TTN with elements $(\mathfrak{D}_{\alpha} \tensor{F})[\mathbb{E}^o \cup \{ i_{|\mathbb{E}^o|+1} \}](\vector{y})$, which can be represented as a TTN diagram with a new open edge extended from the vertex $\alpha$. Note that the total derivative $\tensor{F}'(\vector{y})$ is a $(|\mathbb{E}^o|+1)^{\mathrm{th}}$ order tensor, which can be expressed as a sum of $d$ TTNs using \eqref{eqn:differential-1-var}. However, the resulting tensor $\tensor{F}'(\vector{y})$ may not necessarily retain the TTN structure.

By recursion, we can define higher-order total derivatives of $\tensor{F}$ with respect to $\vector{y}$ as: 
\begin{equation}\label{eqn:high-order-total-derivative}
\begin{aligned}
    \tensor{F}^{(k)}(\vector{y}) &:= \frac{\dd^k}{\dd \vector{y}^k} \left( \tensor{F}(\vector{y}) \right) 
    = \frac{\dd^{k-1}}{\dd \vector{y}^{k-1}} \sum_{\alpha_1} (\mathfrak{D}_{\alpha_1} \tensor{F})(\vector{y}) \\
    &= \cdots 
    = \sum_{\alpha_1,\,\dots,\,\alpha_k} (\mathfrak{D}_{\alpha_k} \cdots \mathfrak{D}_{\alpha_1} \tensor{F})(\vector{y}),
\end{aligned}
\end{equation}
where the summation is over $\alpha_j \in \{1,\,\dots,\,d\}$, $j=1,\,\dots,\,k$. So the summation on the right hand side contains $d^k$ terms. The sequence of $\alpha_1, \dots, \alpha_k$ determines the sequence of differentiation. For different sequences of $\alpha$'s, the resulting partial derivatives may be the same. We do not distinguish between them or calculate their multiplicities here. This issue will be addressed in \Cref{sec:3}. Throughout the derivation process in \eqref{eqn:high-order-total-derivative}, the active contraction indices of each term in the right hand side remain unchanged. As a result, the TTN diagrams of functional TTNs in the summation retain the same structure as that of $\tensor{F}$, except the open edges.

Finally, let us consider a functional TTN, whose components are composite functions. The simplest case is to substitute $\vector{y}$ with $\vector{y}(t)$ in \eqref{eqn:original-tensor}. We then compute the total derivative of $\tensor{F}$ with respect to the scale variable $t$. This derivation process can  be extended to more general situations, where $\vector{y}$ is a vector-valued function of a vector variable. By application of chain rule, we obtain the first order total derivative of $\mathbf{F}$ with respect to $t$ as follows:
\begin{equation}
\label{eq:derivativeforcompositefunctional}
\begin{aligned}
    \frac{\dd}{\dd t} \left( \mathbf{F}(\bm{y}(t)) \right) &:= \frac{\dd}{\dd t} \left( \tensor{F}(\vector{y}(t)) \right) \\
    &= \sum_{\alpha=1} ^d\big((\mathfrak{D}_{\alpha} \tensor{F})(\vector{y}(t))   \big) \;\vector{y}'(t) \\
    &=: \sum_{\alpha=1}^d \tilde{\tensor{F}}_\alpha(\vector{y}(t))
    .
\end{aligned}
\end{equation}
where 
$\tilde{\tensor{F}}_\alpha(\vector{y}(t)) = \big((\mathfrak{D}_{\alpha} \tensor{F})(\vector{y}(t))\big) \;  \vector{y}'(t)$ 
is a TTN with the same number of physical index as $\tensor{F}$, but one more contraction index than $\tensor{F}$, and $\vector{y}'$ can be denoted as the $(d+1)^{\mathrm{th}}$ component of $\tilde{\tensor{F}}_\alpha$. In this case, taking the derivative of a TTN is equivalent to contracting the new physical mode of $\mathfrak{D}_\alpha \tensor{F}$ with $\vector{y}'$, forming a new contraction index. An example illustrating this process is provided in \Cref{fig:example-partial-derivative-TTN-c}. High order derivatives of $\mathbf{F}$ with respect to $t$ can be obtained by recursive application of this process. 

\section{The derivative of tree tensor networks with ODE constraints}\label{sec:3}

Based on the representation of TTNs and their derivatives, we compute the high order derivatives and the Taylor expansion of vector-valued functions satisfying ODE constraints in this section.  
Consider the following ODEs system: 
\begin{equation}\label{eqn:ode-system}
    \vector{y}'(t) = \vector{f}(\vector{y}(t)),
\end{equation}
where $\vector{y}: \mathbb{R} \to \mathbb{R}^d$ is the unknown function and $\vector{f}: \mathbb{R}^d \to \mathbb{R}^d$ is a given functional. This type of ODEs has widely applications in fluid dynamics \cite{chilla2012new, temam2024navier}, nonlinear dynamics and chaos \cite{lorenz1963deterministic, rossler1976equation}, nonlinear optics \cite{agrawal2000nonlinear, akhmediev1997nonlinear}, quantum systems \cite{schrodinger1926undulatory, thaller2013dirac}, plasma physics \cite{chen1984introduction}, condensed matter physics \cite{ginzburg2009theory}, etc. The aim of this subsection is to compute high-order derivatives of $\vector{y}$ with respect to $t$, and represent it in terms of the derivatives of $\vector{f}$ with respect to $\vector{y}$.



\subsection{Differentiation following leaf-wise growth of trees}\label{sec:3.1}
As mentioned in \Cref{TTNderivatives}, when a partial derivative is taken at a vertex in the given TTN, a new edge is introduced and connected from the upper right. If the partial derivative is applied to a composite function, a new vertex is subsequently introduced. The new edge and new vertex are labelled immediately after the respective current largest labels. The $k^\mathrm{th}$ order derivative of $\vector{f}$ respect to $t$ can be recursively written out as follows:
\begin{equation}\label{eqn:high-order-derivative-path-1}
    \begin{aligned}
        \frac{\dd^k}{\dd t^k} \left( \vector{f}(\vector{y}) \right) &= \frac{\dd^{k-1}}{\dd t^{k-1}} \left( (\mathfrak{D}_1 \vector{f}(\vector{y}) ) \vector{f}(\vector{y}) \right) \\
        &= \frac{\dd^{k-2}}{\dd t^{k-2}} \left[ \mathfrak{D}_1 ((\mathfrak{D}_1 \vector{f}(\vector{y}) ) \vector{f}(\vector{y})) \vector{f}(\vector{y}) + \mathfrak{D}_2 ((\mathfrak{D}_1 \vector{f}(\vector{y}) ) \vector{f}(\vector{y})) \vector{f}(\vector{y}) \right] 
        \\
        &= \cdots,
    \end{aligned}
\end{equation}
which is ultimately decomposed into a summation of $k!$ TTNs. This process describes how a TTN grows from the root, adding one leaf at a time, to eventually form a TTN with $k+1$ components.

To write down this process clearly, let us define $\tensor{F}_1 = (\mathfrak{D}_1 \vector{f}(\vector{y}) ) \vector{f}(\vector{y})\in \mathbb{R}^d$ and
\begin{equation}
\label{TTNfordifferentiationpath:eq}
    \tensor{F}_{\alpha_j,\,\ldots,\,\alpha_1} = \left(\mathfrak{D}_{\alpha_j} \tensor{F}_{\alpha_{j-1},\,\ldots,\,\alpha_1}\right) \vector{f}(\vector{y})\in \mathbb{R}^d, \quad \text{for } j > 1.
\end{equation}
Written element-wise, this yields
\begin{equation}
    \begin{aligned}
        \tensor{F}_1[i_1] &= (\mathfrak{D}_1 \vector{f})[i_1, k_1](\vector{y})\; \vector{f}[k_1](\vector{y}), \\
        \tensor{F}_{\alpha_j,\,\ldots,\,\alpha_1}[i_1] &= (\mathfrak{D}_{\alpha_j} \tensor{F}_{\alpha_{j-1},\,\ldots,\,\alpha_1})[i_1, k_j] \;\vector{f}[k_j](\vector{y}) , \quad \text{for } j > 1,
    \end{aligned}
\end{equation}
where $i_1=1,\,2,\,\ldots,\,d$, the subscript $\alpha_i$ denotes the path of differentiation, and the contraction index set for $\tensor{F}_{\alpha_j,\,\ldots,\,\alpha_1}$ is $\{ k_1, \, \ldots, \, k_j \}$. The subscripts $\alpha_j,\,\ldots,\,\alpha_1$ with $\alpha_i\leq i$ are called a {\it differentiation path}, which determines how a tree grows from the root by adding one leaf at a time. Using this notation, \eqref{eqn:high-order-derivative-path-1} can be expressed as follows:
\begin{equation}\label{eqn:high-order-derivative-path-2}
    \begin{aligned}
        \frac{\dd^k}{\dd t^k} \left( \vector{f}(\vector{y}) \right) &= \frac{\dd^{k-1}}{\dd t^{k-1}} \tensor{F}_1  
        = \frac{\dd^{k-2}}{\dd t^{k-2}} \left[ (\mathfrak{D}_1 \tensor{F}_1)\; \vector{f}(\vector{y}) + (\mathfrak{D}_2 \tensor{F}_1) \;\vector{f}(\vector{y})\right] \\
        &= \frac{\dd^{k-2}}{\dd t^{k-2}} \left( \tensor{F}_{1,1} + \tensor{F}_{2,1} \right) \\
        &= \cdots 
        = \sum_{\alpha_1 \leq 1, \, \ldots, \, \alpha_k \leq k} \tensor{F}_{\alpha_k,\,\ldots,\,\alpha_1}.
    \end{aligned}
\end{equation}

For each TTN appearing on the right hand side of \eqref{eqn:high-order-derivative-path-2}, the corresponding TTN configuration can be derived from an unlabelled tree by assigning a contraction index $k_j$ to each corresponding edge. When $k=8$, the configurations of the TTN $\matrix{F}_{1,5,5,3,1,2,1,1}$ and $\matrix{F}_{1,5,5,3,1,1,2,1}$ are displayed in \Cref{fig:valid-labelled-tree-a1} and \Cref{fig:valid-labelled-tree-a2}, while the corresponding unlabelled tree is shown in \Cref{fig:valid-labelled-tree-a0}. To simplify the notation $\tensor{F}_{\alpha_k,\,\ldots,\,\alpha_1}$,  for a given unlabelled tree $T$ defined in \cite[Chapter 3]{butcher2016numerical}, we define a corresponding TTN $T(\vector{f})$, whose TTN diagram closely resembles the graph of $T$.

\begin{figure}[htb]
    \centering
    \subcaptionbox{TTN $\matrix{F}_{1,5,5,3,1,2,1,1}$. \label{fig:valid-labelled-tree-a1}}[0.3\textwidth]{
\ifmytrue
        \begin{forest}
        midtree,
        for tree={%
            s sep=18pt,
        },
        EL/.style 2 args={
            edge label={node[midway, font=\sffamily\notsotiny, #1]{#2}},
        },
        [,fill=none
        [,EL={left=-2pt}{$i_1$}
            [,EL={below right=-4pt}{$k_8$}]
            [,EL={above left=-4pt}{$k_4$}
                [,EL={below right=-4pt}{$k_7$}][,EL={below left=-4pt}{$k_6$}]
            ]
            [,EL={above right=-4pt}{\color{red}$k_2$}
                [,EL={left=-2pt}{$k_5$}]
            ]
            [,EL={below left=-4pt}{$k_1$}
                [,EL={left=-2pt}{\color{red}$k_3$}]
            ]
        ]
        ]
        \end{forest}
\fi
    }
    \hfill
    \subcaptionbox{TTN $\matrix{F}_{1,5,5,3,1,1,2,1}$. \label{fig:valid-labelled-tree-a2}}[0.3\textwidth]{
\ifmytrue
    \begin{forest}
        midtree,
        for tree={%
            s sep=18pt,
        },
        EL/.style 2 args={
            edge label={node[midway, font=\sffamily\notsotiny, #1]{#2}},
        },
        [,fill=none
        [,EL={left=-2pt}{$i_1$}
            [,EL={below right=-4pt}{$k_8$}]
            [,EL={above left=-4pt}{$k_4$}
                [,EL={below right=-4pt}{$k_7$}][,EL={below left=-4pt}{$k_6$}]
            ]
            [,EL={above right=-4pt}{\color{blue}$k_3$}
                [,EL={left=-2pt}{$k_5$}]
            ]
            [,EL={below left=-4pt}{$k_1$}
                [,EL={left=-2pt}{\color{blue}$k_2$}]
            ]
        ]
        ]
        \end{forest}
\fi
    }
    \hfill
    \subcaptionbox{Unlabelled tree $T$. \label{fig:valid-labelled-tree-a0}}[0.3\textwidth]{
\ifmytrue
        \begin{forest}
        midtree
        [,label={[yshift=2pt]}
            [,]
            [,
                [,][,]
            ]
            [,
                [,]
            ]
            [,
                [,]
            ]
        ]
        \end{forest}
\fi
    }
    \caption{TTNs and corresponding unlabelled tree of order 9.}
    \label{fig:TTNandunlabelledtree}
\end{figure}

\begin{definition}[A rooted tree $T$. {\cite[Definition 300B]{butcher2016numerical}} ]\label{def:Rootedtree}
    A triplet $T:=(\mathcal{V},\,\mathcal{E},v_1)$  is a rooted tree if  $(\mathcal{V},\,\mathcal{E})$ is a tree and $v_1\in \mathcal{V}$. The ‘root’ of the rooted tree $T$ is $v_1$.

\end{definition}

    For a rooted tree $T$, let the vertex set be denoted by $\mathcal{V} = \{ v_1, \dots, v_V \}$ and the edge set by $\mathcal{E} = \{ e_1, \dots, e_E \}$. A 7th-order rooted tree is shown as an example in \Cref{fig:example-TTN-Tf-a}. As illustrated in \Cref{fig:example-TTN-Tf-b}, for a given rooted tree $T$, we sometimes add an open edge to the root from below, resulting in a rooted tree with an open edge. In this case, the edge set becomes $\mathcal{E} = \{e_0, e_1, \dots, e_E \}$. 
    The edges incident to vertex $v_\alpha$ are denoted by $\mathcal{E}_\alpha = \mathcal{E}_\alpha^p \cup \mathcal{E}_\alpha^c$, where $\mathcal{E}_\alpha^p$ contains a single edge connecting $v_\alpha$ to its parent vertex, and $\mathcal{E}_\alpha^p$ denotes the set of edges connecting $v_\alpha$ to its child vertices.
    
\begin{definition}[Functional TTN]\label{def:TTN-Tf}
    For a function $\vector{f}(\vector{y}): \mathbb{R}^d \to \mathbb{R}^d$, analytic in a neighbourhood of $\vector{y}$,  we define a corresponding \emph{functional TTN $T(\vector{f})$}, which is itself a vector-valued functional from $\mathbb{R}^d \to \mathbb{R}^d$ and satisfies the following properties:
    \begin{itemize}
        \item[1.] The $T(\vector{f})$ is a TTN consisting $V$ components, whose TTN diagram corresponds to the rooted tree $T$ augmented with an open edge incident to the root from below.
        \item[2.] The $\alpha^\text{th}$ component of $T(\vector{f})$ is denoted by $\vector{f}^{(l_\alpha)}(\vector{y})$, where $l_\alpha=|\mathcal{E}_\alpha^c|$ represents the number of edges incident to the vertex $v_\alpha$ from above (i.e., the number of child edges).
        \item[3.] The first mode of  $\vector{f}^{(l_\alpha)}(\vector{y}) $ corresponds to the edge incident to the parent vertex, while the remaining modes of  $\vector{f}^{(l_\alpha)}(\vector{y}) $ (arising from differentiation) correspond to the edges incident to the child vertices.
        \item[4.] The functional $T(\vector{f})$ is obtained by contracting the functional TTN from the leaves toward the root.
    \end{itemize}
\end{definition}

    For a rooted tree $T$ shown in \Cref{fig:example-TTN-Tf-a}, the TTN diagram of the corresponding TTN $T(\vector{f})$ is illustrated in \Cref{fig:example-TTN-Tf-d}. Since $\vector{f}^{(l_\alpha)}(\vector{y})$ is a symmetry tensor and satisfies \eqref{symmetryvectorpower}, the contraction result is independent of the specific contraction relationships among the indices. Therefore, even without explicitly specifying these contractions, \Cref{def:TTN-Tf} does not lead to any ambiguity.
    
    In order to introduce an algebraic notation for the TTN $T(\vector{f})$, let us denote all contraction indices as $\mathbb{E} = \{ k_1, \dots, k_E \}$ with $ k_j\mapsto e_j$. By replacing $ e_j$ with $k_j$, we obtain the definition of indices configuration $\mathbb{E}_\alpha, \,\mathbb{E}_\alpha^p,\, \mathbb{E}_\alpha^c$, which are corresponding to $\mathcal{E}_\alpha, \,\mathcal{E}_\alpha^p,\, \mathcal{E}_\alpha^c$, respectively. For the root with $\alpha=1$, the physical index is denoted by $i_1$ and the set of physical indices in the root is $\mathbb{E}_1=\{i_1\}$.   
    Using this notations, for a function $\vector{f}: \mathbb{R}^d \to \mathbb{R}^d$, we write the elements of $T(\vector{f})$ as
    \begin{equation}\label{eqn:TTN-Tf-def}
        T(\vector{f})[\mathbb{E}_1^o]:=T(\vector{f})[\mathbb{E}_1^o](\vector{y}) = \sum_{\mathbb{E}} \vector{f}^{(l_1)}[\mathbb{E}_1^o, \mathbb{E}_1](\vector{y}) \prod_{\alpha=2}^{V} \vector{f}^{(l_\alpha)}[\mathbb{E}_\alpha](\vector{y}),
    \end{equation}
    where $l_\alpha = |\mathbb{E}_\alpha|-1$. It is important to note that the TTN $T(\vector{f})$ is uniquely determined by the rooted tree $T$ and the vector-valued function $\vector{f}$. The specific choice of contraction indices does not affect the contraction result of $T(\vector{f})$. This claim will be rigorously proved by \Cref{proposition:TTN-Tf}, \Cref{tensorsymlemma}, and \Cref{lemma:relationshipbetweenTTNandF}.
    
    
    
    

\begin{example}
    We provide a concrete example to illustrate \Cref{def:TTN-Tf}. Suppose $\vector{f}: \mathbb{R}^d \to \mathbb{R}^d$ is a smooth vector-valued function and $T$ is a rooted tree with seven vertices, whose graph is drawn in \Cref{fig:example-TTN-Tf-a}. Then the elements of the functional TTN $T(\vector{f})$ induced by the tree $T$ are given by
    \begin{equation}
    \begin{aligned}
        T(\vector{f})[i_1]:=~&T(\vector{f})[i_1](\vector{y}) = \sum_{k_1, \dots, k_6} \vector{f}^{(3)}[i_1, k_1, k_2, k_3](\vector{y}) \;\vector{f}''[k_1, k_4, k_5](\vector{y})\;\\ \qquad &  \vector{f}[k_2](\vector{y})\; \vector{f}[k_3]\vector{f}'[k_4, k_6](\vector{y})\;\vector{f}[k_5](\vector{y})\; \vector{f}[k_6](\vector{y}).        
    \end{aligned}     
    \end{equation}
    By defining $\mathbb{E}_1 = \{ k_1, k_2, k_3 \}$, $\mathbb{E}_2 = \{ k_1, k_4, k_5 \}$, $\mathbb{E}_3 = \{ k_2 \}$, $\mathbb{E}_4 = \{ k_3 \}$, $\mathbb{E}_5 = \{ k_4, k_6 \}$, $\mathbb{E}_6 = \{ k_5 \}$, and $\mathbb{E}_7 = \{ k_6 \}$, the elements of $T(\vector{f})$ can be written in the general form given in \eqref{eqn:TTN-Tf-def}. 
    The configuration of $T(\vector{f})$ and its corresponding TTN diagram are illustrated in \Cref{fig:example-TTN-Tf-c} and \Cref{fig:example-TTN-Tf-d}, respectively.
\end{example}


\begin{figure}[htb]
    \centering
    \subcaptionbox{A rooted tree $T$ \label{fig:example-TTN-Tf-a}}[0.24\linewidth]{
\ifmytrue
        \begin{forest}
            bigtree
            [,label=below:$v_1$ , EL={left=-2pt}{$n_1$}
                [,label={[yshift=-2pt]above:$v_4$}, EL={right=-2pt}{$e_3$}]
                [,label={[yshift=-2pt]above:$v_3$}, EL={left=-3pt,yshift=2pt}{$e_2$}]
                [,label=left:$v_2$, EL={left=-2pt}{$e_1$}
                    [,label={[yshift=-2pt]above:$v_6$}, EL={right=-2pt}{$e_5$}]
                    [,label=left:$v_5$, EL={left=-2pt}{$e_4$}
                        [,label={[yshift=-2pt]above:$v_7$}, EL={right=-2pt}{$e_6$}]
                    ]
                ]
            ]
        \end{forest}
\fi
    }
    \subcaptionbox{A rooted tree $T$ with an open edge \label{fig:example-TTN-Tf-b}}[0.24\linewidth]{
\ifmytrue
        \begin{forest}
            bigtree
            [,fill=none
            [,label=left:$v_1$ , EL={left=-2pt}{$e_0$}
                [,label={[yshift=-2pt]above:$v_4$}, EL={right=-2pt}{$e_3$}]
                [,label={[yshift=-2pt]above:$v_3$}, EL={left=-3pt,yshift=2pt}{$e_2$}]
                [,label=left:$v_2$, EL={left=-2pt}{$e_1$}
                    [,label={[yshift=-2pt]above:$v_6$}, EL={right=-2pt}{$e_5$}]
                    [,label=left:$v_5$, EL={left=-2pt}{$e_4$}
                        [,label={[yshift=-2pt]above:$v_7$}, EL={right=-2pt}{$e_6$}]
                    ]
                ]
            ]
            ]
        \end{forest}
\fi
    }
    \subcaptionbox{Configuration of $T(\vector{f})$ \label{fig:example-TTN-Tf-c}}[0.24\linewidth]{
\ifmytrue
        \begin{forest}
            bigtree
            [,fill=none 
            [,label=left:$\vector{f}^{(3)}$ , EL={left=-2pt}{$i_1$}
                [,label={[yshift=-2pt]above:$\vector{f}$}, EL={right=-2pt}{$k_3$}]
                [,label={[yshift=-2pt]above:$\vector{f}$}, EL={left=-3pt,yshift=2pt}{$k_2$}]
                [,label=left:$\vector{f}''$, EL={left=-2pt}{$k_1$}
                    [,label={[yshift=-2pt]above:$\vector{f}$}, EL={right=-2pt}{$k_5$}]
                    [,label=left:$\vector{f}'$, EL={left=-2pt}{$k_4$}
                        [,label={[yshift=-2pt]above:$\vector{f}$}, EL={right=-2pt}{$k_6$}]
                    ]
                ]
            ]
            ]
        \end{forest}
\fi
    }
    \subcaptionbox{TTN diagram of $T(\vector{f})$ \label{fig:example-TTN-Tf-d}}[0.24\linewidth]{
\ifmytrue
        \begin{forest}
            bigtree
            [,fill=none 
            [,label=left:$\vector{f}^{(3)}$ , EL={left=-2pt}{$d$}
                [,label={[yshift=-2pt]above:$\vector{f}$}, EL={right=-2pt}{$d$}]
                [,label={[yshift=-2pt]above:$\vector{f}$}, EL={left=-3pt}{$d$}]
                [,label=left:$\vector{f}''$, EL={left=-2pt}{$d$}
                    [,label={[yshift=-2pt]above:$\vector{f}$}, EL={right=-2pt}{$d$}]
                    [,label=left:$\vector{f}'$, EL={left=-2pt}{$d$}
                        [,label={[yshift=-2pt]above:$\vector{f}$}, EL={right=-2pt}{$d$}]
                    ]
                ]
            ]
            ]
        \end{forest}
\fi
    }
    \caption{An example of functional TTN $T(\vector{f})$.}
    \label{fig:example-TTN-Tf}
\end{figure}



We now present an alternative perspective for interpreting $T(\vector{f})$, which will be used to further reveal the relationship between $T(\vector{f})$ and the TTNs appearing on the right-hand side of \eqref{eqn:high-order-derivative-path-2}. To this end, we first recall the recursive definition of rooted trees.

\begin{definition}[{\cite[III.1 Definition 1.1]{hairer2006geometric}}]
    The set of (rooted) trees $\mathbb{T}$ is recursively defined as follows:
    \begin{enumerate}[(a)]
        \item The graph $\;\begin{forest} [] \end{forest} = [1]$ with only one vertex (called the root) belongs to $\mathbb{T}$;
        \item If $T_1, \ldots,  T_m \in \mathbb{T}$, then the graph obtained by grafting the roots of $T_1, \ldots, T_m$ to a new vertex also belongs to $\mathbb{T}$. It is denoted by 
        \begin{equation*}
            T = [T_1 \dots T_m],
        \end{equation*}
        and the new vertex is the root of $T$.
    \end{enumerate}
\end{definition}

With this definition of trees, the `elementary differentials' is defined as follows.

\begin{definition}[{\cite[Definition 310A]{butcher2016numerical}}]
    Given a tree $T$ and a function $\vector{f}: \mathbb{R}^d \to \mathbb{R}^d$, analytic in a neighbourhood of $\vector{y}$, the `elementary differential' $\vector{F}(T)(\vector{y})$ is defined by 
    \begin{equation}
        \begin{aligned}
            \vector{F}([1])(\vector{y}) &= \vector{f}(\vector{y}), \\
            \vector{F}([T_1 \dots T_m])(\vector{y}) &= \vector{f}^{(m)}(\vector{y}) \left(\vector{F}(T_1)(\vector{y}),\, \dots,\, \vector{F}(T_m)(\vector{y}) \right).
        \end{aligned}
    \end{equation}
\end{definition}

In the above definition, the term $\vector{f}^{(m)}(\vector{y}) \left(\vector{F}(T_1)(\vector{y}),\, \dots,\, \vector{F}(T_m)(\vector{y}) \right)$ is understood by viewing $\vector{f}^{(m)}(\vector{y})$ as a $m$-linear operator acting on $m$ vectors $\vector{F}(T_1)(\vector{y}),\, \dots,\,\allowbreak \vector{F}(T_m)(\vector{y})$. Elementary differentials can be interpreted as TTNs of $\vector{f}$, with the underlying TTN diagram induced by the tree $T$. This is revealed by the following proposition.

\begin{proposition} 
\label{proposition:TTN-Tf}
For a function $\vector{f}: \mathbb{R}^d \to \mathbb{R}^d$, analytic in a neighbourhood of $\vector{y}$, according to \eqref{symmetryproperty}, the TTN $T(\vector{f})\in\mathbb{R}^d$ satisfies the following recursive properties:
\begin{equation} \label{eqn:TTN-Tf}
\begin{aligned}
    T(\vector{f}) &= \vector{f}(\vector{y})\in\mathbb{R}^d,\quad\hbox{ if } T=[1],\\ 
    T(\vector{f}) &:= \vector{f}^{(l)}(\vector{y})\times^1 T_1(\vector{f})\times^1 T_2(\vector{f})\cdots \times^1 T_l(\vector{f})\\ 
    &=\vector{f}^{(l)}(\vector{y})\, T_1(\vector{f}) T_2(\vector{f})\cdots T_l(\vector{f}) \in\mathbb{R}^d,\quad\hbox{ if } T=[T_1 T_2\dots T_l].
\end{aligned}
\end{equation}
\end{proposition}

\begin{proof}
    The first property follows directly from \Cref{def:TTN-Tf} when $T=[1]$. The second property is based on the fact that TTN $T(\vector{f})$ is constructed by contracting components from the leaves toward the root.
    Since $T = [T_1 \dots T_l]$, we first contract the edges within the subtrees $T_1,\, \dots,\, T_l$, yielding TTNs $T_1(\vector{f}),\, \dots, \, T_l(\vector{f})$, respectively. We then contract these TTNs with the root component $\vector{f}^{(l)}(\vector{y})$, which establishes the second property. 
\end{proof}


Next, we study the relationship between TTN $T(\vector{f})$ and TTNs $\tensor{F}_{\alpha_k,\,\ldots,\,\alpha_1}$ in \eqref{eqn:high-order-derivative-path-2}.
Due to the partial symmetry of $\vector{f}^{(l)}(\vector{y})$ as stated in \eqref{symmetryproperty}, we obtain symmetry property  of TTN $T(\vector{f})$ in \Cref{tensorsymlemma}.


\begin{lemma}
\label{tensorsymlemma}
    For an unlabelled tree $T=[T_1 \dots T_l]$ and any permutation $\pi$ of $\{1,\,\ldots,$ $\,l\}$, the corresponding TTN $T(\vector{f})$ satisfies:
    \begin{equation}
    \label{tensorsymlemma:eq1}
    T(\vector{f})=\vector{f}^{(l)}(\vector{y}) T_1(\vector{f}) \cdots T_l(\vector{f})=\vector{f}^{(l)}(\vector{y}) T_{\pi(1)}(\vector{f}) \cdots T_{\pi(l)}(\vector{f}).
\end{equation}
    
\end{lemma}
The proof of \Cref{tensorsymlemma} follows directly from \eqref{symmetryproperty}. 

\begin{lemma}
\label{lemma:relationshipbetweenTTNandF}
    For any TTN $\tensor{F}_{\alpha_k,\,\ldots,\,\alpha_1}$ corresponding to an unlabelled tree $T$ of order $k+1$, the following relationship holds: 
    \begin{equation}
    \label{lemma:relationshipbetweenTTNandF:eq1}
    \tensor{F}_{\alpha_k,\,\ldots,\,\alpha_1} = T(\vector{f}).
\end{equation}
    
\end{lemma}

\begin{proof}
    We prove this statement by mathematical induction. For $k=1$, the conclusion follows directly. Now, assume that the statement holds for any unlabelled tree with order less than $k+1$. For a tree $T$ with $|T|=k+1$, we express it as $T=[T_1 \dots T_l]$, where $T_i$ for $i=1,\,\ldots,\,l$ are subtrees of $T$ of order less than $k+1$. We contract $\tensor{F}_{\alpha_k,\,\ldots,\,\alpha_1}$ from top to bottom. Based on the induction hypothesis, for the $i^{\mathrm{th}}$ subtree, we obtain $T_i(\vector{f})\in\mathbb{R}^d$. Suppose the contraction order in $\tensor{F}_{\alpha_k,\,\ldots,\,\alpha_1}$ at the root vertex is given by a permutation $\pi$, then we have
\begin{equation}
    \tensor{F}_{\alpha_k,\,\ldots,\,\alpha_1} = \vector{f}^{(l)}(\vector{y}) T_{\pi(1)}(\vector{f}) \cdots T_{\pi(l)}(\vector{f})=\vector{f}^{(l)}(\vector{y}) T_1(\vector{f}) \cdots T_l(\vector{f})=T(\vector{f}),
\end{equation}
which completes the proof of the lemma.
\end{proof}

The \Cref{lemma:relationshipbetweenTTNandF} describes the `commutativity' of tensor partial derivatives. For example, $\matrix{F}_{1,5,5,3,1,2,1,1}=\matrix{F}_{1,5,5,3,1,1,2,1}$ indicates that the second and third partial derivatives can be interchanged. In \cite[Chapter 3]{butcher2016numerical}, a ‘forest' is defined as a collection of trees, potentially with repetitions, such as $T_1\, T_2 \dots T_l$. In tensor notation, a forest is represented as the product of $T_1(\vector{f})$, $ \ldots ,\,T_l(\vector{f})\in\mathbb{R}^d$, which is subsequently contracted with a symmetric tensor of order greater than or equal to $l$, as exemplified in \eqref{tensorsymlemma:eq1}. For $T=[T_1 \dots T_l]$, the Butcher product of trees $T$ and $T_{l+1}$, defined as $\tilde{T}=T\circ {T_{l+1}}:=[T_1 \dots T_l\, T_{l+1}]$, can be  expressed in terms of TTNs as follows:
\begin{equation}
    \tilde{T}(\vector{f}):={T}(\vector{f})\circ{T}_{l+1}(\vector{f})=\vector{f}^{(l+1)}(\vector{y}) T_1(\vector{f}) \cdots T_l(\vector{f})T_{l+1}(\vector{f}).
\end{equation}
If trees are denoted by TTNs $\tensor{F}_{\alpha_k,\,\ldots,\,\alpha_1}$ and ${\tensor{F}}_{\tilde{\alpha}_l,\,\ldots,\,\tilde{\alpha}_1}$, the Butcher product can be represented as:
\begin{equation}
    \label{Butchertreeproduct}
    \tensor{F}_{\alpha_k,\,\ldots,\,\alpha_1} \circ \tensor{F}_{\tilde{\alpha}_l,\,\ldots,\,\tilde{\alpha}_1} = \tensor{F}_{\tilde{\alpha}_l+k+1,\,\ldots,\,\tilde{\alpha}_1+k+1,\,1,\,\alpha_k,\,\ldots,\,\alpha_1}.
\end{equation}

Next, let us introduce several definitions and functions about unlabelled tree $T$.
\begin{definition}[{\cite[Theorem 304A and 305A]{butcher2016numerical}}] For an unlabelled tree $T = [T_1^{m_1}\, \allowbreak T_2^{m_2} \dots T_k^{m_k}]$, where $T_1,\,T_2,\,\ldots,\,T_k$ are distinct trees. Then we define
\begin{equation}
\begin{aligned}
    |T| &= 1 + \sum_{i=1}^km_i|T_i|,\\
    \sigma(T) & = \prod_{i=1}^km_i!\sigma(T_i)^{m_i},\\
    \gamma(T)&=T!= |T|  \prod_{i=1}^k (T_i!)^{m_i},\\
    \alpha(T)&=\frac{|T|!}{\sigma(T)T!}.
\end{aligned}
\end{equation}
\end{definition}

According to \Cref{lemma:relationshipbetweenTTNandF}, different differentiation paths result in distinct TTN configurations, yet they may correspond to the same unlabelled tree.  Motivated by this observation, we introduce the concept of a labelled tree to simplify TTN configurations and systematically analyze differentiation paths. In a labelled tree, the root vertex is assigned the value $1$, and each vertex connected from below by an edge with index $k_j$ is assigned by the value $j+1$. An example of such a labelled tree is provided in \Cref{fig:valid-labelled-tree}. Since the labelled trees are simplified representations of TTN configurations generated by differentiation, they must satisfy the following property:

(1) Each vertex receives one and only one label from $\{1,\,2,\,\ldots,\,|T|\}$.

(2) The vertices connected to a given vertex from above are assigned labels in increasing order from left to right.

(3) If $ (i,j)$ is a labelled edge then $i<j$.

Labelled trees satisfying these three properties are called \emph{valid labelled trees}. The definition of valid labelled trees here is closely related to the labelled trees defined in \cite[Chapter 3]{butcher2016numerical}. In \cite{butcher2016numerical}, the second property is replaced by: equivalent labellings under the symmetry group are counted only once. In this work, we select a specific labelled tree from the symmetry group using criterion (2). For example, $\begin{forest}
        midtree
        [,label={[yshift=0pt]left:$1$}
            [,label={right:$3$}]
            [,label={left:$2$}
            ]
        ]
        \end{forest}$ is chosen from the symmetry group $\left\{\begin{forest}
        midtree
        [,label={[yshift=0pt]left:$1$}
            [,label={right:$3$}]
            [,label={left:$2$}
            ]
        ]
        \end{forest},\begin{forest}
        midtree
        [,label={[yshift=0pt]left:$1$}
            [,label={right:$2$}]
            [,label={left:$3$}
            ]
        ]
        \end{forest}\right\}$.

\begin{theorem}
\label{numberofvalidlabelledtree}
The number of distinct ways of labelling the given tree $T$ under these conditions (1)-(3) is $\alpha(T)$.
\end{theorem}
\begin{proof}
    By removing the condition (2), we obtain all labelled trees satisfying (1) and (3) within the symmetry group $\mathbb{A}(T)$. For a labelled tree in $\mathbb{A}(T)$ with $l$ subtrees denoted as $T_1,\,\ldots,\,T_l$, assume the root of $T_i$ is assigned a value $a_i$. There exists a permutation $\pi$ such that $a_{\pi(i)}$ is in increasing order. Then the tree $T=[T_{\pi(1)}\ldots T_{\pi(l)}]$ belongs to $\mathbb{A}(T)$ and satisfies condition (2) at the first layer. By repeating this process, we eventually obtain a labelled tree $\hat{T}$ that satisfies conditions (1)-(3) in $\mathbb{A}(T)$. To prove the uniqueness, assume there exist two distinct labelled trees $\hat{T}_1$ and $\hat{T}_2$ satisfying (1)-(3). Since the root of both $\hat{T}_1$ and $\hat{T}_2$ is assigned the value 1, and both trees belong to $\mathbb{A}(T)$, the sets of labelled values in the first layer of $\hat{T}_1$ and $\hat{T}_2$ must be identical. Furthermore, by condition (2), we find that the topologies of $\hat{T}_1$ and $\hat{T}_2$ in the first layer are also identical. By induction, each subsequent layer of $\hat{T}_1$ and $\hat{T}_2$ must also be identical, implying $\hat{T}_1=\hat{T}_2$. Thus, the conditions (1)-(3) are equivalent to those used in \cite{butcher2016numerical}. By \cite[Theorem 305A]{butcher2016numerical}, the number of distinct ways of labelling the given tree $T$ is given by $\alpha(T)$.
\end{proof}

\begin{theorem}
\label{OnetoonemapbetweenlabelledtreeandTTN}
There exists a bijection between valid labelled trees with order $k+1$ and differentiation paths ${\alpha_k,\,\ldots,\,\alpha_1}$, where $\alpha_i\leq i$.
\end{theorem}
\begin{proof}
    For a labelled tree $T$, let $\mathbb{E}$ denote the set of all edges $(i,j)$, ordered such that $i<j$. For any $j\in\{1,\,2,\,\ldots,k\}$, if $(i,j+1)\in\mathbb{E}$, let us define $\varphi_T(j)=i$. Due to $j+1>i$, we have $\varphi_T(j)\leq j$. Since each $j\in\{1,\,2,\,\ldots,k\}$ corresponds to a unique parent vertex $i$ (connected from below to the vertex labelled $j+1$), and because $T$ follows an upward-growing representation, the function $\varphi_T(j)$ is well-defined and satisfies $\varphi_T(j)\leq j$. Moreover, distinct labelled trees $T_1\neq T_2$ yield distinct functions $\varphi_{T_1}\neq \varphi_{T_2}$. 
    
    Define a mapping $\phi$ that assigns each valid labelled tree $T$ of order $k+1$ to the sequence ${\varphi_T(k),\,\ldots,\,\varphi_T(1)}$. It is easy to verify that the mapping $\phi$ is well-defined. Since $\varphi_T(j)\leq j$,  the sequence ${\varphi_T(k),\,\ldots,\,\varphi_T(1)}$ represents a valid differentiation path. The injectivity of $\phi$ follows directly from the fact that $\varphi_{T}$ is distinct for different trees $T$.
    We now prove that $\phi$ is surjective by induction. For the base case, when $k=1$, the differentiation path $\alpha_1=1$ clearly corresponds to a valid labelled tree of order 2. Next, let us construct a preimage $T$ for any given differentiation path ${\alpha_{k-1},\,\ldots,\,\alpha_1}$. Assume as the induction hypothesis that there exists a labelled tree $\tilde{T}$ of order $k$ being preimage of the differentiation path ${\alpha_{k-1},\,\ldots,\,\alpha_1}$. Let $\varphi_{\tilde{T}}:\,\{1,\,2,\,\ldots,k-1\}\rightarrow \{1,\,2,\,\ldots,k-1\}$  be the associated mapping. To construct a tree $T$ of order $k+1$ corresponding to the extended differentiation path ${\alpha_k,\,\alpha_{k-1},\,\ldots,\,\alpha_1}$, we attach a new leaf with label $k+1$ to the $(\alpha_k)^{\mathrm{th}}$ vertex of $\tilde{T}$.
    The resulting mapping $\varphi_T$ is well-defined as
    $$
    \varphi_T(j)=\left\{\begin{aligned}
        \varphi_{\tilde{T}}(j),\,j<k,\\
        \alpha_k,~~~\,j=k.
    \end{aligned} 
    \right.
    $$
    Thus, the labelled tree $T$ is the preimage of the differentiation path ${\alpha_k,\,\alpha_{k-1},\,\ldots,\,\alpha_1}$, and $\phi$ is a surjective.
    Since $\phi$ is both injective and surjective, it is a bijection, completing the proof.
\end{proof}

\begin{corollary}
    As there are $k!$ distinct different differentiation paths, and $a_{k+1}$ counts valid labelled trees of order $k+1$ (satisfying conditions (1)-(3)),  the bijection implies: $$a_{k+1}=\sum_{|T|=k+1}\alpha(T)=k!.$$
\end{corollary}

According to \Cref{lemma:relationshipbetweenTTNandF}, two differentiation paths are equivalent if there corresponding to the same unlabelled tree. It is also interesting to obtain this equivalence from the `commutativity' between tensor partial derivatives or `commutativity' in the sequence of differentiation path. 
To investigate this, we define a modified Butcher product of trees as
$$
[T_1] \hatcirc [T_2] = [T_1\,T_2]=[T_1] \circ T_2,
$$ 
which means that the root of $[T_2]$ is removed, and the subtree $T_2$ is attached to the root of $[T_1]$. Similarly, we have 
$$ [T]\hatcirc[\tilde{T}_1\, \tilde{T}_2 \dots \tilde{T}_l] = [T\, \tilde{T}_1 \dots \tilde{T}_l] = [T]\hatcirc[\tilde{T}_1]\hatcirc[\tilde{T}_2] \hatcirc \cdots \hatcirc[\tilde{T}_l]. 
$$
When the trees are represented as TTNs $\tensor{F}_{\alpha_k,\,\ldots,\,\alpha_1}$ and ${\tensor{F}}_{\tilde{\alpha}_l,\,\ldots,\,\tilde{\alpha}_1}$, the modified Butcher product can be represented as:
\begin{equation}
    \label{modifiedButchertreeproduct}
    \tensor{F}_{\alpha_k,\,\ldots,\,\alpha_1} \hatcirc \tensor{F}_{\tilde{\alpha}_l,\,\ldots,\,\tilde{\alpha}_1}=\tensor{F}_{\hat{\alpha}_l,\,\ldots,\,\hat{\alpha}_1,\,\alpha_k,\,\ldots,\,\alpha_1},
\end{equation}
where
$
\hat{\alpha}_i = \left\{\begin{aligned}
    &1 ,\quad \hbox{if }\, \tilde{\alpha}_i=1,\\
    &\tilde{\alpha}_i+ k,\quad \hbox{else if }\,\tilde{\alpha}_i>1.
\end{aligned}\right.
$

For $i_0\in[1,\,k]$, the modified Butcher product $\hatcirc_{i_0} $ can be represented as:
\begin{equation}
    \label{modifiedButchertreeproduct-1}
    \tensor{F}_{\alpha_k,\,\ldots,\,\alpha_1}\hatcirc_{i_0}\,{\tensor{F}}_{\tilde{\alpha}_l,\,\ldots,\,\tilde{\alpha}_1}=\tensor{F}_{\hat{\alpha}_l,\,\ldots,\,\hat{\alpha}_1,\,\alpha_k,\,\ldots,\,\alpha_1},
\end{equation}
where
$
\hat{\alpha}_i = \left\{\begin{aligned}
    &i_0 ,\quad \hbox{if }\, \tilde{\alpha}_i=1,\\
    &\tilde{\alpha}_i+ k,\quad \hbox{else if }\,\tilde{\alpha}_i>1.
\end{aligned}\right.
$
The modified Butcher product $\hatcirc_{i_0} $ means that we remove the root of the second TTN and connect all sub-TTNs to the $i_0^{\mathrm{th}}$ leaf of the first one. It is clear that $\hatcirc=\hatcirc_{1}$. 

\begin{lemma}
For any tree $T=[T_1\ldots T_l]$, we have 
\begin{equation}
T=[T_1]\circ T_2\circ \cdots \circ T_l=[T_1]\hatcirc [T_2]\hatcirc \cdots \hatcirc [T_l],
\end{equation}
and 
\begin{equation}
T(\vector{f})=\left(\vector{f}'(\vector{y})T_1(\vector{f})\right)\circ T_2(\vector{f})\circ \cdots \circ T_l(\vector{f}).
\end{equation}
Furthermore, assuming that the TTNs for trees $T$, $T_1,\,\ldots,\,T_l$, are $\tensor{F}$,  $\tensor{F}^1,\,\ldots,$ and $\tensor{F}^l $, we get
\begin{equation}
\label{decompositionofbutcherproduct}
\tensor{F} = \hat{\tensor{F}} \circ \tensor{F}^1 \circ \cdots \circ \tensor{F}^l,
\end{equation}
where $\hat{\tensor{F}}$ corresponds to unlabelled tree $\bullet $.
\end{lemma}

The proof of this lemma follows directly from the definition of the Butcher product of trees. We then utilize the (modified) Butcher product to establish the equivalence of differentiation paths, thereby avoiding the need to explicitly construct the underlying unlabelled tree. For a given TTN $\tensor{F}_{\alpha_k,\,\ldots,\,\alpha_1}$, let us move all indices with the largest value to the left. In this procedure, since we only adjust the labels of vertices not involved in differentiation, the topology of the corresponding unlabelled tree remains unchanged. Denote $\beta=\max_{i=1}^k\alpha_i$, and suppose there are $s$ indices $\alpha_i$ equal to $\beta$. We can write down this procedure using modified Butcher product as:
$$
\tensor{F}_{\alpha_k,\,\ldots,\,\alpha_1}=\tensor{F}_{\tilde{\alpha}_{k-s},\,\ldots,\,\tilde{\alpha}_1} \hatcirc_{\beta-1}\tensor{F}_{11\ldots1},
$$
where $\tensor{F}_{11 \ldots 1}$ corresponds to an $s+1$ order unlabelled tree and the subscripts  $\beta-1$ indicate that the root of $\tensor{F}_{11\ldots1}$ is generated by taking the  $(\beta-1)^{\mathrm{th}}$ partial derivative $\partial_{\tilde{\alpha}_{\beta-1}}$ in the differentiation path ${\tilde{\alpha}_{k-s},\,\ldots,\,\tilde{\alpha}_1}$. Here, the sequence ${\tilde{\alpha}_{k-s},\,\ldots,\,\tilde{\alpha}_1}$ is obtained by removing all entries $\alpha_i=\beta$ from the sequence ${\alpha_k,\,\ldots,\,\alpha_1}$. 
By applying this process recursively, we obtain the decomposition given in \eqref{decompositionofbutcherproduct}. Let us use $\tensor{F}_{1,5,5,3,1,1,2,1}$ as an example to illustrate this process. 
$$
\begin{aligned}
\tensor{F}_{1,5,5,3,1,2,1,1}&=\tensor{F}_{1,3,1,2,1,1} \hatcirc_{4}\tensor{F}_{11}\\
&=\tensor{F}_{1,1,2,1,1}\hatcirc_{2}\tensor{F}_{1} \hatcirc_{4}\tensor{F}_{11}=\tensor{F}_{1,1,1,1}\hatcirc_{1}\tensor{F}_{1}\hatcirc_{2}\tensor{F}_{1} \hatcirc_{3}\tensor{F}_{11}\\
&=\tensor{F}_{1}\circ\tensor{F}_1\circ\tensor{F}_{1} \circ\tensor{F}_{11}.    
\end{aligned}
$$
During this procedure, the subscript of $\hatcirc$ should be updated according to the updated number of leaves.

\Cref{OnetoonemapbetweenlabelledtreeandTTN} implies that valid labelled trees uniquely determine the differentiation path. By combining \Cref{lemma:relationshipbetweenTTNandF}, \Cref{numberofvalidlabelledtree}, and \Cref{OnetoonemapbetweenlabelledtreeandTTN}, we conclude that there exist exactly $\alpha(T)$ differentiation paths to obtain the TTN $T(\vector{f})$. Finally, we compute high-order derivatives of $\vector{y}$ in the following theorem.

\begin{theorem}
\label{theorem-exact-high-order-derivative}
Assume that $\vector{y}$ is a solution of ODEs system \eqref{eqn:ode-system}. Then, the $(k+1)^\mathrm{th}$ order derivative of $\vector{y}$ has the following expression 
    \begin{equation}\label{eqn:ode-expansion-a}
    \vector{y}^{(k+1)} =  \sum_{|T|=k+1} \alpha(T) T(\vector{f}).
\end{equation}
\end{theorem}

\begin{proof}
    Using \eqref{eqn:high-order-derivative-path-2}, we have
    \begin{equation}
        \vector{y}^{(k+1)}=\sum_{\alpha_1 \leq 1, \, \ldots, \, \alpha_k \leq k} \tensor{F}_{\alpha_k,\,\ldots,\,\alpha_1},
\end{equation}
which represents $\vector{y}^{(k+1)}$ as a summation of $k!$ TTNs. According to \Cref{lemma:relationshipbetweenTTNandF}, \Cref{OnetoonemapbetweenlabelledtreeandTTN}, and the multiplicity of $T(\vector{f})$, we simplify $\vector{y}^{(k+1)}$ as:
  \begin{equation}
    \vector{y}^{(k+1)} =  \sum_{|T|=k+1} \alpha(T) T(\vector{f}),
\end{equation}
which completes the proof.
\end{proof}
With the definition of $T(\vector{f})$ and the TTN's derivatives introduced in \Cref{TTNderivatives}, we can express any high order derivative of $\vector{y}$ as the summation of TTNs $T(\vector{f})$. This result is formally stated in \Cref{theorem-exact-high-order-derivative}, which is identical to Theorem 311C in \cite[Chapter 3]{butcher2016numerical}. In that context, $T(\vector{f})$ corresponds to the ‘elementary differential' as defined in Definition 310A of \cite[Chapter 3]{butcher2016numerical}. However, it is important to note that while our final result aligns with that of \cite{butcher2016numerical}, our proof methodology follows a different approach.

\subsection{Differentiation following layer-wise growth of trees}
\label{subsec:Differentiation following layer-wise growth of trees}
In this subsection, we utilize TTN derivatives to introduce a layer-wise growth concept and provide an alternative proof for \Cref{theorem-exact-high-order-derivative} based on this insight.
Similar to \eqref{eq:derivativeforcompositefunctional}, taking the $k^{\mathrm{th}}$ order derivative with respect to $t$ on both sides of \eqref{eqn:ode-system} and using the chain rule, the $(k+1)^\mathrm{th}$ order derivative of $\vector{y}$ has the following expression:
\begin{equation}\label{eqn:high-order-derivative-y-general-expression}
    \vector{y}^{(k+1)} = (\vector{f}(\vector{y}))^{(k)} = \sum \vector{f}^{(l)}(\vector{y}) \; \vector{y}^{(j_1)} \cdots \vector{y}^{(j_l)},
\end{equation}
where $l\leq k$, $j_1 + \cdots + j_l = k$ with $j_i\geq 1$.
Here the summation runs over all TTNs generated  recursively using \eqref{eq:derivativeforcompositefunctional}. For $k=0,\,1,\,2,\,3,\,4$, the total number of TTNs in this summation is 1, 1, 2, 5, and 15, respectively. The structures of these TTNs are determined by the multiple indices $j_1,\ldots,j_l$ and some may share the same index combinations. For a given $k$, these TTNs can be visualized using TTN diagrams, as shown in \Cref{fig:tree-tensor-network-ode-constraint}. Due to the partial symmetric property \eqref{symmetryproperty}, the ordering of indices $j_1,\ldots,j_l$ does not affect the final contraction result. Consequently, different feasible index choices for $j_1,\ldots,j_l$, such as $\vector{f}^{(3)}(\vector{y}) \; \vector{y}' \vector{y}'' \vector{y}^{(3)}$ and $\vector{f}^{(3)}(\vector{y}) \; \vector{y}' \vector{y}^{(3)} \vector{y}''$, may lead to identical TTNs. We do not explicitly account for the two types of multiplicity of these TTNs at this stage; instead, we incorporate it after substituting the constraint $\vector{y}'=\vector{f}(\vector{y})$ in \eqref{eqn:high-order-derivative-y-general-expression}.



\begin{figure}[htb]
    \centering
    \subcaptionbox{}{
\ifmytrue
    \begin{forest}
    midtree
       [,fill=none  [,label=right:$\vector{f}(\vector{y})$
        ] ]
    \end{forest}
\fi
    }
    $\rightarrow$
    \subcaptionbox{}{
\ifmytrue
    \begin{forest}
    midtree
       [,fill=none  [,label=right:$\vector{f}'(\vector{y})$
            [,label=above:$\vector{y}'$]
        ] ]
    \end{forest}
\fi
    }
     $\rightarrow$ 
     \subcaptionbox{}{
\ifmytrue
     \begin{forest}
    midtree
       [,fill=none  [,label=right:$\vector{f}''(\vector{y})$
            [,label=above:$\vector{y}'$]
            [,label=above:$\vector{y}'$]
        ] ]
    \end{forest} 
\fi
     }
    \subcaptionbox{}{
\ifmytrue
    \begin{forest}
    midtree
       [,fill=none  [,label=right:$\vector{f}'(\vector{y})$
            [,label=above:$\vector{y}''$]
        ] ]
    \end{forest}
\fi
    }
    $\rightarrow$ 
    \subcaptionbox{}{
\ifmytrue
    \begin{forest}
    midtree
       [,fill=none  [,label=right:$\vector{f}^{(l)}$
            [,label=above right:$\vector{y}^{(j_{l})}$]
            [,label={[xshift=4pt]above:$\vector{y}^{(j_{l-1})}$}]
            [$\ldots$,fill=none,no edge]
            [,label={[xshift=5pt]above:$\vector{y}^{(j_2)}$}]
            [,label={[xshift=10pt]above left:$\vector{y}^{(j_1)}$}]
        ] ]
    \end{forest}
\fi
    }
    \caption{The TTN diagrams of TTNs in the summation of $\vector{y}^{(k+1)}$. (a) $k=0$; (b) $k=1$; (c-d) $k=2$; (e) $k>2$.}
    \label{fig:tree-tensor-network-ode-constraint}
\end{figure}

By recursively substituting the lower order derivatives of $\vector{y}$ into \eqref{eqn:high-order-derivative-y-general-expression}, we can finally represent $\vector{y}^{(k+1)}$ as a summation of TTNs $T(\vector{f})$. This recursive process is equivalent to the growth of a tree, layer by layer, from its root. 
Using the notation $T(\vector{f})$, we further expand \eqref{eqn:high-order-derivative-y-general-expression} as
\begin{equation}\label{eqn:ode-expansion}
    \vector{y}^{(k+1)} = (\vector{f}(\vector{y}))^{(k)} = \sum_{|T|=k+1} \tilde{\alpha}(T) T(\vector{f}),
\end{equation}
where the TTN $ T(\vector{f})\in\mathbb{R}^d $ and the scale $ \tilde{\alpha}(T)$ represents the multiplicity of $T(\vector{f})$. 
This procedure offers an insight into the layer-wise growth of trees. By replacing $ \tilde{\alpha}(T)$ with ${\alpha}(T)$, we can complete the proof of \Cref{theorem-exact-high-order-derivative}.
Next, we introduce an alternative method to determine it, based on the layer-wise growth insight of the tree.

To identify the multiplicity $\tilde{\alpha}(T)$, we employ the symmetry property \eqref{symmetryproperty} and reorganize the terms $\vector{y}^{(j_i)}$ in ascending order of their indices $j_1,\,\ldots,\,j_l$. This ordering procedure transforms \eqref{eqn:high-order-derivative-y-general-expression} into the canonical form:
\begin{equation}
\label{eqn:high-order-derivative-y-multiplicity}
    (\vector{f}(\vector{y}))^{(k)} = \sum_{\vector{m} \in \cup_{l=1}^{k} S_{k,l}} \eta(\vector{m}) \vector{f}^{(l)}(\vector{y}) (\vector{y}')^{m_1} (\vector{y}'')^{m_2} \cdots (\vector{y}^{(k)})^{m_k},
\end{equation}
where $\eta(\vector{m}) \in \mathbb{R}$ is the combinatorial multiplicity factor to be determined and the index sets are defined as:
\begin{equation}
\label{m-set}
S_{k,l} = \left\{ \vector{m} = (m_1,\, \dots,\, m_k) \in \mathbb{N}_0^{k} \Big| \sum_{i=1}^k m_i = l,\, \sum_{i=1}^k i m_i = k \right\}.
\end{equation}
Here $m_i$ represents the multiplicity of $\vector{y}^{(i)}$ in each term of the summation of \eqref{eqn:high-order-derivative-y-general-expression}, with $m_i=0$ indicating the absence of $\vector{y}^{(i)}$ in that term. In \eqref{eqn:high-order-derivative-y-multiplicity} and throughout the remainder of this proof, we employ the compact notation defined in \eqref{symmetryvectorpower} to represent contractions between tensors and vectors.


To determine $\eta(\vector{m})$, we observe that its values remain unchanged for both scalar-valued and vector-valued functions $\vector{y}$. To derive its explicit form, we consider the scalar case with the Taylor expansion: 
\begin{equation*}
    y(t) \equiv \sum_{i=1}^{k} y^{(i)}(0) t^{i} / i!.
\end{equation*}
For the monomial function $f(y)=y^{l_0}$, where $f^{(l)}(y)|_{t=0}={\delta_{l_0,l}}{l!}$, we compute: 
\begin{equation}\label{eqn:special-y}
    \begin{split}
        \frac{\dd^k}{\dd t^k} & (f(y)) \big|_{t=0} 
        = \frac{\dd^k}{\dd t^k} \left[ \left( \sum_{i=1}^{k} y^{(i)}(0) \frac{t^i}{i!} \right)^{l_0} \right] \Bigg|_{t=0} \\
        &= \sum_{\vector{m} \in  S_{k,l_0}}  k!{\binom{l_0}{\vector{m}} } \left(\frac{y'(0)}{1!}\right)^{m_1} \left(\frac{y''(0)}{2!}\right)^{m_2} \cdots \left(\frac{y^{(k)}(0)}{k!}\right)^{m_k} \\
        &= \sum_{\vector{m} \in \cup_{l=1}^{k} S_{k,l}} \delta_{l_0,l} k!{\binom{l}{\vector{m}} } \left(\frac{y'(0)}{1!}\right)^{m_1} \left(\frac{y''(0)}{2!}\right)^{m_2} \cdots \left(\frac{y^{(k)}(0)}{k!}\right)^{m_k} \\
        &= \sum_{\vector{m} \in \cup_{l=1}^{k} S_{k,l}}  \frac{\binom{l}{\vector{m}} k!}{(1!)^{m_1} (2!)^{m_2} \cdots (k!)^{m_k}} \frac{f^{(l)}(y)\big|_{t=0}}{l!} (y'(0))^{m_1} \cdots (y^{(k)}(0))^{m_k},
    \end{split}
\end{equation}
where $\delta$ denotes the Kronecker delta.
Due to the arbitrariness of $y^{(i)}(0)$, by comparing \eqref{eqn:high-order-derivative-y-multiplicity} and  \eqref{eqn:special-y}, we obtain 
\begin{equation}\label{eqn:eta}
    \eta(\vector{m}) =\frac{\tilde{\eta}(\vector{m})}{\hat{\eta}(\vector{m})}\quad\hbox{ for } \vector{m}\in  S_{k,l_0},
\end{equation}
where
\begin{equation}
\begin{aligned}
    \tilde{\eta}(\vector{m}) &= \frac{k!}{(1!)^{m_1} (2!)^{m_2} \cdots (k!)^{m_k} },\\
    \hat{\eta}(\vector{m})&={m_1! m_2! \cdots m_k!}.
    \end{aligned}
\end{equation}

Changing $l_0\in\{1,\,2,\,\ldots\}$, we confirm that \eqref{eqn:eta} holds for all $\vector{m}\in S_{k,l}$,
which coincides exactly with the Fa\`a di Bruno's formula \cite[pages 35-37]{riordan2014introduction}. The function $\eta$, given by \eqref{eqn:eta}, represents the contribution of taking derivatives at the root layer to the multiplicity $\alpha(T)$. 

We proceed by induction to complete the proof for \Cref{theorem-exact-high-order-derivative}. Assume that \eqref{eqn:ode-expansion-a} holds for $\vector{y}^{(j)}$ for $j<k+1$. According to \eqref{eqn:high-order-derivative-y-multiplicity}, we have 
\begin{equation}
\label{eq:theorem-5}
   \vector{y}^{(k+1)} = \sum_{\vector{m} \in \cup_{l=1}^{k} S_{k,l}} \eta(\vector{m}) \vector{f}^{(l)}(\vector{y}) \prod_{i=1}^{k}\left(\sum_{|\tilde{T}|=i}\alpha(\tilde{T})\tilde{T}(\vector{f})\right)^{m_i},
\end{equation}
where $\tilde{T}(\vector{f})\in\mathbb{R}^d$. 
Assume that all distinct subtrees of $T$ have different orders, i.e., $|T_i|\neq |T_j|$ for $i\neq j $, $|T_i|>0$, and $|T_j|>0$. In this case, we can express $T$ as $T=[T_1^{m_1}\, T_2^{m_2}\cdots T_k^{m_k}]$, which leads to
\begin{equation}
\label{eq:theorem-1}
    T(\vector{f}) =  \vector{f}^{(l)}(\vector{y}) (T_1(\vector{f}))^{m_1} (T_2(\vector{f}))^{m_2} \cdots (T_k(\vector{f}))^{m_k}.
\end{equation}
By comparing \eqref{eq:theorem-5}, \eqref{eq:theorem-1}, and \eqref{eqn:ode-expansion}, we obtain
\begin{equation}
\label{eq:theorem-3}
\begin{aligned}
       \tilde{\alpha}(T) &= \eta(\vector{m})\prod_{i=1}^k \left(\alpha(T_i)\right)^{m_i}\\
       &=\frac{|T|!}{|T|\prod_{i=1}^k (i!)^{m_i}}\frac{\prod_{i=1}^k (|T_i|!)^{m_i}}{\prod_{i=1}^k (T_i!)^{m_i}}\frac{1}{\hat{\eta}({\vector{m}})\prod_{i=1}^k \sigma(T_i)^{m_i}}\\
       &=\frac{|T|!}{T!}\frac{1}{\sigma(T)}=\alpha(T).
\end{aligned}
\end{equation}
Now, suppose that some distinct subtrees of $T$ share the same order, i.e., $|T_i|= |T_j|>0$ for $i\neq j $. In this case, we denote the tree $T$ as $$T=[T_{1,1}^{m_{1,1}} \dots T_{1,n_1}^{m_{1,n_1}}\, T_{2,1}^{m_{2,1}}\dots T_{2,n_2}^{m_{2,n_2}} \ldots T_{k,1}^{m_{k,1}} \dots T_{k,n_k}^{m_{k,n_k}}],$$
and express its TTN representation as
\begin{equation}
\label{eq:theorem-2}
    T(\vector{f}) =  \vector{f}^{(l)}(\vector{y}) \prod_{j=1}^{n_1}(T_{1,j}(\vector{f}))^{m_{1,j}} \prod_{j=1}^{n_2}(T_{2,j}(\vector{f}))^{m_{2,j}} \cdots \prod_{j=1}^{n_k}(T_{k,j}(\vector{f}))^{m_{k,j}},
\end{equation}
where $\sum_{j=1}^{n_i} m_{i,j}=m_i$ and $|T_{i,j}| = |T_{i,j'}|=i$ hold for $i=1,\,2,\,\ldots,\,k$ and $j\neq j'$. Let us denote $\vector{m}_i=(m_{i,1},\,\ldots,\,m_{i,n_i})$. For any symmetric tensor $\matrix{B}\in\mathbb{R}^{d\times \cdots \times d}$ of order greater than $m_i$, we have
\begin{equation}
\begin{aligned}
       \matrix{B}(\vector{y}^{(i)})^{m_i}&= \matrix{B}\left(\sum_{|\tilde{T}_i|=i}\alpha(\tilde{T}_i)\tilde{T}_i(\vector{f})\right)^{m_i}\\&
       = {\binom{m_i}{\vector{m}_i} }\left(\prod_{j=1}^{n_i} \left(\alpha(T_{i,j})\right)^{m_{i,j}}\right)\matrix{B}\prod_{j=1}^{n_i} \left(T_{i,j}(\vector{f})\right)^{m_{i,j}} + \cdots.  
\end{aligned}
\end{equation}
Therefore, $\tilde{\alpha}(T)$ is computed as
\begin{equation}
\begin{aligned}
\label{eq:theorem-4}
       \tilde{\alpha}(T) &= \eta(\vector{m})\prod_{i=1}^k {\binom{m_i}{\vector{m}_i} } \left(\prod_{j=1}^{n_i} \left(\alpha(T_{i,j})\right)^{m_{i,j}}\right)\\
       &=\frac{|T|!}{|T|\prod_{i=1}^k (i!)^{m_i}}\frac{\prod_{i=1}^k \prod_{j=1}^{n_i} (|T_{i,j}|!)^{m_{i,j}}}{\prod_{i=1}^k \prod_{j=1}^{n_i}(T_{i,j}!)^{m_{i,j}}\sigma(T_{i,j})^{m_{i,j}}}\frac{\prod_{i=1}^k {\binom{m_i}{\vector{m}_i} }}{\hat{\eta}({\vector{m}})}\\
       &=\frac{|T|!}{|T|\prod_{i=1}^k (i!)^{m_i}}\frac{\prod_{i=1}^k (i!)^{m_i}}{\prod_{i=1}^k \prod_{j=1}^{n_i}(T_{i,j}!)^{m_{i,j}}}\frac{1}{\prod_{i=1}^k \prod_{j=1}^{n_i}m_{i,j}! \sigma(T_{i,j})^{m_{i,j}}}\\
       &=\frac{|T|!}{T!}\frac{1}{\sigma(T)}=\alpha(T),
\end{aligned}
\end{equation}
where the fact  $|T_{i,j}| =i$ is used. 

By combining \eqref{eq:theorem-3} and \eqref{eq:theorem-4}, we complete the proof of \Cref{theorem-exact-high-order-derivative} via mathematical induction.

\subsection{Labelling trees}\label{subsubsec:labelling-trees}

In \Cref{numberofvalidlabelledtree}, we established that the number of valid labelled trees corresponding to a given unlabelled tree $T$ is $\alpha(T)$, utilizing the enumeration of labelled trees provided in \cite[Theorem 305A]{butcher2016numerical}. In this subsection, we present an alternative, independent approach to counting the number of valid labelled trees.

For any given tree $T$, the generation of tree can be viewed hierarchically. For example, suppose that $T$ with labelled values ranging from $1$ to $9$ is a $9^\mathrm{th}$ order tree (e.g. \Cref{fig:valid-labelled-tree}), arising from the computation of $\vector{y}^{(9)}$. According to \eqref{eqn:high-order-derivative-y-multiplicity}, the first laryer of $T$ is generated with the structure shown in \Cref{subfig:computing-alpha-2(a)}. At the same time, the smallest label is assigned to the root, and the rest of the labels are partitioned into disjoint subsets, where the number of elements in each subset is determined by the order of the corresponding subtree (which is the order of derivatives taken on $\vector{y}$). A possible labelling scheme for $T$ at this stage is illustrated in \Cref{subfig:computing-alpha-2(b)}.


\begin{figure}[htb]
    \centering
    \subcaptionbox{Valid labelled tree for $\matrix{F}_{1,5,5,3,1,2,1,1}$: hierarchical separation of sets. \label{fig:valid-labelled-tree}}[0.3\textwidth]{
\ifmytrue
        \begin{forest}
        midtree
        [,label={[yshift=2pt]below:$1$}
            [,label={right:$9$}]
            [,label={[xshift=2pt]left:$5$}
                [,label=above:$8$][,label=above:$7$]
            ]
            [,label={[xshift=3pt]left:$3$}
                [,label=above:$6$]
            ]
            [,label={left:$2$}
                [,label=above:$4$]
            ]
        ]
        \end{forest}
\fi
    }
    \hfill
    \subcaptionbox{First layer of $T$ with the leaf denoted by $\vector{y}^{(k_j)}$.  \label{subfig:computing-alpha-2(a)} }[0.3\textwidth]{
\ifmytrue
        \begin{forest}
        bigtree
            [,fill=none
            [,label={[yshift=-3pt]right:$\vector{f}^{(4)}$}
                [,label=above:$\vector{y}'$][,label=above:$\vector{y}'''$][,label=above:$\vector{y}''$][,label=above:$\vector{y}''$]
            ]
            ]
        \end{forest}
\fi        
    }
    \hfill
    \subcaptionbox{A labelling scheme for $T$ at the first stage. \label{subfig:computing-alpha-2(b)} }[0.3\textwidth]{
\ifmytrue
        \begin{forest}
        bigtree
            [,fill=none
            [,label={[yshift=-3pt]right:{\small $1$}}
                [,label={[xshift=9pt]above:{\small $\{9\}$}}]
                [,label={[xshift=4pt]above:{\small $\{5,7,8\}$}}]
                [,label={[xshift=-5pt]above:{\small $\{3,6\}$}}]
                [,label={[xshift=-9pt]above:{\small $\{2,4\}$}}]
            ]
            ]
        \end{forest}
\fi
    }
    \caption{A labelled tree in the summation of $\vector{y}^{(9)}$.}
    \label{fig:computing-alpha-2}
\end{figure}

In general, when computing $\vector{y}^{(k+1)}$ with \eqref{eqn:high-order-derivative-y-general-expression}, if the underlying tree structure is $T = [T_1 \, \ldots \, T_l]$ with $|T| = k+1$, the smallest label is assigned to the root, and the rest $k$ labels are divided into $l$ sets, with size $|T_1|,\, \dots,\, |T_l|$, respectively.
To count the number of all valid labelled trees, let us define
\begin{equation*}
    \mathbf{K} = \big[ \overbrace{k_1,\, \dots,\, k_1}^{e_1},\, \overbrace{k_2,\, \dots,\, k_2}^{e_2},\, \dots,\, \overbrace{k_j,\,\dots,\,k_j}^{e_j} \big],
\end{equation*}
which is a rearrangement of $[|T_1|, \,\dots,\,|T_l|]$ by grouping identical elements together. Here $1\leq k_i \leq k$ denotes the order of the derivative of $\vector{y}$ in \eqref{eqn:high-order-derivative-y-general-expression}, and $e_1,\,\ldots,\,e_j$ correspond to a rearrangement of all nonzero elements of $\vector{m}$. According to \eqref{m-set}, we have $|T| = k_1 e_1 + \cdots + k_j e_j + 1$, $l = e_1 + \cdots e_j$, and
\begin{equation*}
    \mathbf{K}! = (k_1!)^{e_1} (k_2!)^{e_2} \cdots (k_j!)^{e_j}.
\end{equation*}
Then, the number of ways of partitioning $|T|-1$ different labels into $l$ distinguishable sets, corresponding to $l$ subtrees connected with the root, whose sizes are $\mathbf{K}$, is 
\begin{equation*}
    \tilde{\eta}(\mathbf{K}) = \frac{k!}{\mathbf{K}!}.
\end{equation*}
Proceed recursively, the sets of labels in $\vector{y}^{(k_1)},\, \vector{y}^{(k_2)},\, \dots,\, \vector{y}^{(k_j)}$ undergo the above process respectively. In this sense, trees can be viewed as a hierarchical separation of sets (\Cref{fig:valid-labelled-tree}). During this process, each node of the tree structure receives a label, which is the smallest number in the current label set. The final result is a tree whose leaves are labelled. The number of ways of partitioning is 
\begin{equation*}
    \tilde{\eta}(T)=\frac{|T|!}{T!}.
\end{equation*}

In this procedure, all sets (subtrees) are temporarily treated as distinguishable at each stage, even when they are isomorphic. This artificial distinguishability leads to overcounting, as each tree $T$ is enumerated $\sigma(T)$ times (where $\sigma(T)$ is the order of its symmetry group \cite[page 154]{butcher2016numerical}). After compensating for this overcounting, we obtain the final count of valid labelled trees $\alpha(T)$ for a given tree $T$ as
\begin{equation*}
    \alpha(T) = \frac{|T|!}{\sigma(T) T!}.
\end{equation*}

\subsection{Taylor expansion}

To derive the Taylor expansion of $\vector{y}(t)$ satisfying the ODEs system \eqref{eqn:ode-system}, it is equivalent to computing the derivatives $\vector{y}^{(k)}$ for $k = 1,\,2,\,3,\, \dots\,$. Using \Cref{theorem-exact-high-order-derivative}, we obtain the Taylor expansion of $\vector{y}(t)$, stated formally in the following theorem.

\begin{theorem}
\label{theorem:taylorexpansionfory}
    The Taylor expansion of $\vector{y}(t)$ satisfying the ODEs system \eqref{eqn:ode-system} at $t = t_0$ is given by:
\begin{equation}\label{eqn:ode-taylor}
    \vector{y}(t) = \vector{y}(t_0) + \sum_{k=1}^{\infty} \frac{1}{k!} (t-t_0)^k \sum_{|T|=k} \alpha(T) \left(T(\vector{f})\right)(\vector{y}(t_0)).
\end{equation}
The equality holds for $t$ within a neighborhood of $t_0$ where the Taylor series converges.
\end{theorem}

\section{The derivative of tree tensor networks with constraints driven by Runge--Kutta method} \label{subsec:constraint-driven-by-runge-kutta}

RK method is a well-known single step method for numerically solving ODEs system \eqref{eqn:ode-system} \cite{butcher2015runge,hairer1987solving,hairer2006numerical,humphries1994runge,verwer1996explicit}. In each step of an RK method, the stage values and the next step solution can all be viewed as functions of step size $h$. 

\subsection{The derivative of vector-valued functions with algebraic constraints}\label{sec:The derivative of vector-valued functions with algebraic constraints}
With some abstraction, we consider the following constraint which captures the core feature of the RK method:
\begin{equation}\label{eqn:RK-constraint}
    \vector{y} = \vector{y}_0 + h \vector{f}(\vector{y}), 
\end{equation}
where $\vector{y} = \vector{y}(h) \in \mathbb{R}^d$ is a function of $h$, and $\vector{y}_0:=\vector{y}(0)$ or $\vector{y}_0:=\vector{y}(t_0)$ is independent of $h$. The $k^{\mathrm{th}}$ order derivative of $\vector{y}$ at $h=0$ is given by the following theorem.  

\begin{theorem}
\label{theorem-algebraic-high-order-derivative}
Assume that $\vector{y}$ is a solution of the algebraic system \eqref{eqn:RK-constraint}. Then, the $k^\mathrm{th}$ order derivative of $\vector{y}$ at $h=0$ is given by 
    \begin{equation}\label{eqn:RK-expansion}
    \vector{y}^{(k)} |_{h=0} = \sum_{|T| = k} \gamma(T) \alpha(T) T(\vector{f}) \big|_{h=0}.
\end{equation}
\end{theorem}

\begin{proof}
    Using the Leibniz rule, the high order derivative of $\vector{y}$ with respect to $h$ is
\begin{equation*}
    \vector{y}^{(k)} = (h \vector{f}(\vector{y}))^{(k)} = k (\vector{f}(\vector{y}))^{(k-1)} + h (\vector{f}(\vector{y}))^{(k)}.
\end{equation*}
Using \eqref{eqn:high-order-derivative-y-multiplicity}, the $(k-1)^\mathrm{th}$ order derivative of $\vector{f}(\vector{y})$ is 
\begin{equation*}
    (\vector{f}(\vector{y}))^{(k-1)} = \sum_{\vector{m} \in \cup_{l=1}^{k-1} S_{k-1,l}} \eta(\vector{m}) \vector{f}^{(l)}(\vector{y}) (\vector{y}')^{m_1} (\vector{y}'')^{m_2} \cdots (\vector{y}^{(k-1)})^{m_{k-1}}.
\end{equation*}
Thus, 
\begin{equation*}
    \vector{y}^{(k)} = k \left[ \sum_{\vector{m} \in \cup_{l=1}^{k-1} S_{k-1,l}} \eta(\vector{m}) \vector{f}^{(l)}(\vector{y}) (\vector{y}')^{m_1} (\vector{y}'')^{m_2} \cdots (\vector{y}^{(k-1)})^{m_{k-1}} \right] + h (\vector{f}(\vector{y}))^{(k)}.
\end{equation*}
Taking the limit $h\to 0$, we have
\begin{equation*}
    \vector{y}^{(k)} \big|_{h=0} = k  \sum_{\vector{m} \in \cup_{l=1}^{k-1} S_{k-1,l}} \eta(\vector{m}) \left[\vector{f}^{(l)}(\vector{y}) (\vector{y}')^{m_1} (\vector{y}'')^{m_2} \cdots (\vector{y}^{(k-1)})^{m_{k-1}} \right]_{h=0}.
\end{equation*}

Recursively applying the above procedure, one eventually gets
\begin{equation}\label{eqn:RK-expansion-a}
    \vector{y}^{(k)} |_{h=0} = \sum_{|T| = k} \gamma(T) \alpha(T) T(\vector{f}) \big|_{h=0},
\end{equation}
where $\gamma(T)$ is the multiplicity of $\alpha(T)T(\vector{f})$ introduced by the use of Leibniz rule.

\end{proof}

The value of $\gamma(T)$ is computed by contracting a particular TTN, where each vertex is assigned a value $k$ if the subtree rooted at that vertex has $k$ vertices, as shown in \Cref{fig:representation-numerical-tree-c}.

\begin{remark*}
{\it
    The constraint \eqref{eqn:RK-constraint} can be transformed into ODE constraints with a similar form as \eqref{eqn:ode-system}. By differentiating both sides of \eqref{eqn:RK-constraint} with respect to $h$, we obtain
    \begin{equation*}
        \vector{y}' = \vector{f}(\vector{y}) + h \vector{f}'(\vector{y}) \vector{y}'.
    \end{equation*}
    For small enough $h$, this gives 
    \begin{equation*}
        \vector{y}' = \left(\mathbf{I}  - h\vector{f}'(\vector{y}) \right)^{-1} \vector{f}(\vector{y}).
    \end{equation*}
    With the definition $\vector{Y} = \begin{bmatrix} \vector{y} \\ h \end{bmatrix}$, the above equation is equivalent to 
    \begin{equation*}
        \vector{Y}' = \vector{F}(\vector{Y}),
    \end{equation*}
    where 
    \begin{equation*}
        \vector{F}(\vector{Y}) := \begin{bmatrix}
            \left( \mathbf{I}  - h\vector{f}'(\vector{y}) \right)^{-1} \vector{f}(\vector{y}) \\ 1
        \end{bmatrix}.
    \end{equation*}
    By employing this transformation, we reformulate \eqref{eqn:RK-constraint} as ODE constraints. Following the procedures outlined in \Cref{sec:3}, we can derive derivatives of  $\vector{Y}$ to any desired order. However, since the computation of high order derivatives of $\vector{F}$ with respect to $\vector{Y}$ is rather complicated, we do not elaborate on this procedure in this paper.
    }
\end{remark*}

\begin{remark*}
{\it
    The findings derived in this subsection can be extended to other general algebraic constraints, such as,
    \begin{equation}\label{eqn:extended-constraint}
        \vector{y}' = \vector{y}_0 + p(h) \vector{f}(\vector{y}),
    \end{equation}
    where $\vector{y} = \vector{y}(h)$ is still a vector-valued function of $h$, and $p(h)$ is some known polynomial of $h$. Due to the separation of $h$ and $\vector{y}$ on the right hand side of \eqref{eqn:extended-constraint}, the process of computing high order derivatives of $\vector{y}$ that satisfying constraints \eqref{eqn:extended-constraint} is quite similar to that of $\vector{y}$ satisfying constraints \eqref{eqn:RK-constraint}. 
    }
\end{remark*}



\subsection{Taylor expansion with tensor algebraic constraints}

We consider algebraic constraints driven by the RK methods and compute the Taylor expansion of a vector-valued function subject to these constraints. This forms a fundamental component in establishing the order conditions of RK methods.

Let $\vector{Y}_i$, $i=1,\,\dots,\,s$, be vector-valued functions of $h$. Consider 
\begin{equation}\label{eqn:RK-stages}
    \vector{Y}_i = \vector{y}_0 + h \sum_{j=1}^{s} a_{ij} \vector{f}(\vector{Y}_j), \quad i = 1, \dots, s,
\end{equation}
where $\vector{y}_0:=\vector{y}(0)$ or $\vector{y}_0:=\vector{y}(t_0) \in \mathbb{R}^d$, $\vector{Y}_i: \mathbb{R} \to \mathbb{R}^d$, and $\vector{f}: \mathbb{R}^d \to \mathbb{R}^d$.

Rewrite the system \eqref{eqn:RK-stages} in the following matrix form
\begin{equation*}
    \begin{bmatrix} \vector{Y}_1 \\ \vdots \\ \vector{Y}_s \end{bmatrix}
    = \begin{bmatrix} \vector{y}_0 \\ \vdots \\ \vector{y}_0 \end{bmatrix} 
    + h \begin{bmatrix}
        a_{11} \matrix{I} & a_{12} \matrix{I} & \cdots & a_{1s} \matrix{I} \\
        \vdots   & \vdots   &        & \vdots   \\
        a_{s1} \matrix{I} & a_{s2} \matrix{I} & \cdots & a_{ss} \matrix{I} 
    \end{bmatrix} 
    \begin{bmatrix} \vector{f}(\vector{Y}_1) \\ \vdots \\ \vector{f}(\vector{Y}_s) \end{bmatrix},
\end{equation*}
or in a more compact notation
\begin{equation}\label{eqn:RK-stages-matrix-compact}
    \vector{Y} = \vector{Y}_0 + h \Tilde{\matrix{A}} \vector{F}(\vector{Y}),
\end{equation}
where $\vector{Y}: \mathbb{R} \to \mathbb{R}^{sd}$, $\vector{Y}_0 \in \mathbb{R}^{sd}$, $\Tilde{\matrix{A}} = \matrix{A} \otimes \matrix{I} \in \mathbb{R}^{(sd) \times (sd)}$, and $\vector{F}: \mathbb{R}^{sd} \to \mathbb{R}^{sd}$. 

By denoting $\Tilde{\vector{F}} (\vector{Y}) = \Tilde{\matrix{A}} \vector{F}(\vector{Y})$, we represent the constraint \eqref{eqn:RK-stages-matrix-compact} as:
\begin{equation*}
    \vector{Y} = \vector{Y}_0 + h \Tilde{\vector{F}}(\vector{Y}),
\end{equation*}
which is the constraint we have discussed in \Cref{sec:The derivative of vector-valued functions with algebraic constraints}. Using \eqref{eqn:RK-expansion}, we obtain
\begin{equation}\label{eqn:algebraic-constraints-expansion}
    \vector{Y}^{(k)} \big|_{h=0} = \sum_{|T|=k} \gamma(T) \alpha(T) T(\Tilde{\vector{F}}) |_{h=0},
\end{equation}
where $T(\tilde{\vector{F}}) \in \mathbb{R}^{sd}$.

To simplify notation, let us introduce a TTN denoted as $\Phi(T)$. The TTN $\Phi(T)$ shares the same structure of TTN diagram as that as $T$. For each vertex in the TTN diagram of $\Phi(T)$ with $l+1$ connected edges, the corresponding component tensor is $\matrix{A} \matrix{I}_s^l$, where $\matrix{I}_s^l = \operatorname{diag}_l (1,\,\ldots,\,1) \in \mathbb{R}^{s\times \cdots \times s}$ is a diagonal tensor with $l$ nonzero elements, each equal to $1$ (as defined in \cite{kolda2009tensor}). It is evident that $\Phi(T)$ represents a contraction of method-dependent tensors determined by the RK tableau. In \Cref{fig:representation-numerical-tree-b}, a specific TTN $T$ is used to illustrate the definition of $\Phi(T)$. Using this notation, we derive the Taylor expansion for $\vector{Y}$ satisfying \eqref{eqn:RK-stages-matrix-compact}, as stated in the following theorem. 

\begin{theorem}
The Taylor expansion of $\vector{y}(h)$ satisfying the constraints \eqref{eqn:RK-stages} at $h = 0$ is given by:
\begin{equation}
    \vector{Y}(h) = \vector{Y}_0 + \sum_{k=1}^{\infty} \frac{h^k}{k!} \sum_{|T| = k} \alpha(T) \gamma(T) \Phi(T) \otimes T(\vector{f})(\vector{y}_0).
\end{equation}
The equality holds for $h$ within a neighborhood of $0$ where the Taylor series converges.
\end{theorem}

\begin{proof}
    Let us use tensor product notations to compute $T(\tilde{\vector{F}})$ as introduced in \eqref{eqn:algebraic-constraints-expansion} and reveal its relationship with $T(\vector{f})$. According to $\Tilde{\vector{F}}(\vector{Y}) = \Tilde{\matrix{A}} \vector{F}(\vector{Y})$, for $l=0,\,1,\,2,\,\ldots,$ we obtain
\begin{equation*}
    \Tilde{\vector{F}}^{(l)} (\vector{Y}) = \Tilde{\matrix{A}} \vector{F}^{(l)}(\vector{Y}),
\end{equation*}
which is the basic component of the TTN $T (\Tilde{\vector{F}})$. For each $l$, the TTN diagram of $\tilde{\vector{F}}^{(l)}(Y)$ is displayed in \Cref{fig:basic-component}. Substituting the basic components $ \Tilde{\matrix{A}} \vector{F}^{(l)}$ in the TTN $T(\Tilde{\vector{F}})$ (e.g. \Cref{fig:representation-numerical-tree-a}), we get a new TTN with components being $\tilde{\tensor{A}}$ and $\vector{F}^{(l)}$ (see \Cref{fig:representation-numerical-tree-b} for an example).

\begin{figure}[htb]
    \centering
    \subcaptionbox{$l=0$}[0.1\linewidth]{
    \begin{forest}
        midtree
        [,fill=none [,label=right:$\Tilde{\matrix{A}}$ [,label=right:$\vector{F}$ [,no edge,fill=none]]] ]
    \end{forest}
    }\quad
    \subcaptionbox{$l=1$}[0.1\linewidth]{
    \begin{forest}
        midtree
        [,fill=none [,label=right:$\Tilde{\matrix{A}}$ [,label=right:$\vector{F}'$ [,rectangle,fill=none] ]] ]
    \end{forest}
    }\quad
    \subcaptionbox{$l=2$}{
    \begin{forest}
        midtree
        [,fill=none [,label=right:$\Tilde{\matrix{A}}$ [,label=right:$\vector{F}''$ [,rectangle,fill=none][,rectangle,fill=none] ]] ]
    \end{forest}
    }\quad
    \subcaptionbox{$l=3$}{
    \begin{forest}
        midtree
        [,fill=none [,label=right:$\Tilde{\matrix{A}}$ [,label=right:$\vector{F}^{(3)}$ [,rectangle,fill=none][,rectangle,fill=none][,rectangle,fill=none] ]] ]
    \end{forest}
    }\quad
    $\cdots$
    \subcaptionbox{$l$}{
    \begin{forest}
        midtree
        [,fill=none [,label=right:$\Tilde{\matrix{A}}$ [,label=right:$\vector{F}^{(l)}$ [,rectangle,fill=none][,rectangle,fill=none][{$\ldots$},no edge,rectangle,fill=none][,rectangle,fill=none][,rectangle,fill=none] ]] ]
    \end{forest}
    }
    \caption{Basic components of TTN $T(\Tilde{\mathbf{F}})$.}
    \label{fig:basic-component}
\end{figure}

\begin{figure}[htb]
    \centering
    \subcaptionbox{Abstract form of $T (\Tilde{\vector{F}})\in\mathbb{R}^{sd}$ \label{fig:representation-numerical-tree-a}}[0.45\textwidth]{
        \begin{forest}
        midtree
        [,fill=none
            [,label=right:$\Tilde{\vector{F}}''$ [,label=right:$\Tilde{\vector{F}}'$ [,label=right:$\Tilde{\vector{F}}$ ]]  [,label=left:$\Tilde{\vector{F}}$ ] ]
        ]
        \end{forest}
    }
    \hfill
    \subcaptionbox{Explicit representation of $T (\Tilde{\vector{F}})\in\mathbb{R}^{sd}$ \label{fig:representation-numerical-tree-b}}[0.5\textwidth]{
        \begin{forest}
            midtree
            [,fill=none 
                [,label=right:$\Tilde{\matrix{A}}$ [,label=right:$\vector{F}''$ 
                    [,label=right:$\Tilde{\matrix{A}}$ [,label=right:$\vector{F}'$ 
                            [,label=right:$\Tilde{\matrix{A}}$ [,label=right:$\vector{F}$]]
                    ]]  
                    [,label=left:$\Tilde{\matrix{A}}$ [,label=left:$\vector{F}$]]
                ]]
            ]
        \end{forest}
    }
    \par
    \subcaptionbox{$\gamma(T)\in\mathbb{R}$ \label{fig:representation-numerical-tree-c}}[0.25\textwidth]{
        \begin{forest}
        midtree
        [,label=right:$4$ [,label=right:$2$ [,label=right:$1$ ]]  [,label=left:$1$ ] ]
        \end{forest}
    }
    \subcaptionbox{$\Phi(T)\in\mathbb{R}^s$ \label{fig:representation-numerical-tree-d}}[0.3\textwidth]{
        \begin{forest}
            midtree
            [, fill=none
            [,label=right:$\matrix{A}$ [,label=right:$\matrix{I}_s^3$ [,label=right:$\matrix{A}$ [,label=right:$\matrix{I}_s^2$ [,label=right:$\matrix{A}$ [,label=right:$\vector{1}_s$] ]]]  [,label=left:$\matrix{A}$ [,label=left:$\vector{1}_s$] ]] ]
            ]
        \end{forest}
    }
    \subcaptionbox{$T(\vector{f})\in\mathbb{R}^d$ \label{fig:representation-numerical-tree-e}}[0.3\textwidth]{
        \begin{forest}
        midtree
        [, fill=none 
            [,label=right:$\vector{f}''$ [,label=right:$\vector{f}'$ [,label=right:$\vector{f}$ ]]  [,label=left:$\vector{f}$ ] ]
        ]
        \end{forest}
    }
    \caption{Representation of $T(\Tilde{\vector{F}})$.}
    \label{fig:representation-numerical-tree}
\end{figure}


Since $\vector{F}(\vector{Y}) = \begin{bmatrix} \vector{f}(\vector{Y}_1) \\ \vdots \\ \vector{f}(\vector{Y}_s) \end{bmatrix}$, by taking the limit $h \to 0$, we have
\begin{equation*}
    \lim_{h \to 0} \vector{F}(\vector{Y}) = \begin{bmatrix}
        \vector{f}(\vector{y}_0) \\ \vdots \\ \vector{f}(\vector{y}_0)
    \end{bmatrix} = \vector{1}_s \otimes \vector{f}(\vector{y}_0) =: \mathbf{1}_s \otimes \vector{f} |_{h=0},
\end{equation*}
where $\mathbf{1}_s = \begin{bmatrix}
    1 \\ \vdots \\ 1
\end{bmatrix} \in \mathbb{R}^s$. Similarly, 
\begin{equation*}
    \lim_{h \to 0} \vector{F}^{(l)}(\vector{Y}) = \matrix{I}_s^{l+1} \otimes \vector{f}^{(l)}(\vector{y}_0) =: \matrix{I}_s^{l+1} \otimes \vector{f}^{(l)} |_{h=0}, \quad \text{for } l \geq 1.
\end{equation*}
Due to $\Tilde{\matrix{A}} = \matrix{A} \otimes \matrix{I}$, we have
\begin{equation}\label{eqn:RK-AF}
    \lim_{h \to 0} \Tilde{\matrix{A}} \vector{F}(\vector{Y}) = (\matrix{A} \otimes \matrix{I})(\vector{1}_s \otimes \vector{f}|_{h=0}) = (\matrix{A} \vector{1}_s) \otimes (\matrix{I}  \vector{f}|_{h=0}) = (\matrix{A} \vector{1}_s) \otimes \vector{f}|_{h=0},
\end{equation}
and
\begin{equation}\label{eqn:RK-AFl}
    \lim_{h \to 0} \Tilde{\matrix{A}} \vector{F}^{(l)}(\vector{Y}) = (\matrix{A} \otimes \matrix{I})(\matrix{I}_s^{l+1} \otimes \vector{f}^{(l)}|_{h=0}) = (\matrix{A} \matrix{I}_s^{l+1}) \otimes (\matrix{I} \vector{f}^{(l)}|_{h=0}) = (\matrix{A} \matrix{I}_s^{l+1}) \otimes \vector{f}^{(l)}|_{h=0}.
\end{equation}
Here, we applied the mixed product property introduced in \Cref{lemma:mixed-product-property}. Using \eqref{eqn:RK-AF}, \eqref{eqn:RK-AFl}, the mixed product property and associativity, the TTN $T(\Tilde{\vector{F}})$ can be fully decoupled as
\begin{equation}\label{eqn:decomposition-phi-T}
    \lim_{h \to 0} T(\Tilde{\vector{F}}) = \Phi(T) \otimes T(\vector{f})|_{h=0},
\end{equation}
where $\Phi(T) \in \mathbb{R}^{s}$ is a TTN independent of $\vector{f}$. 
Equation \eqref{eqn:decomposition-phi-T} establishes the connection between $T(\tilde{\vector{F}})$ and $T(\vector{f})$, showing that $T(\tilde{\vector{F}})$ is the Kronecker product of a method-dependent TTN $\Phi(T)$ and $T(\vector{f})$. In \Cref{fig:representation-numerical-tree}, a specific TTN $T$ is used to illustrate the connection between $T(\tilde{\vector{F}})$ and $T(\vector{f})$.

Using \eqref{eqn:algebraic-constraints-expansion} and \eqref{eqn:decomposition-phi-T}, the $k^\mathrm{th}$ order derivative of $\vector{Y}$ at $h=0$ can be computed by
\begin{equation*}
    \vector{Y}^{(k)} \big|_{h=0} = \sum_{|T|=k} \alpha(T) \gamma(T) \Phi(T) \otimes T (\vector{f})(\vector{y}_0).
\end{equation*}
This decomposition simplify the calculation of $\vector{Y}^{(k)}$. In summary, the Taylor expansion of $\vector{Y}$ at $h = 0$ can be written as:
\begin{equation*}
    \vector{Y}(h) = \vector{Y}_0 + \sum_{k=1}^{\infty} \frac{h^k}{k!} \sum_{|T| = k} \alpha(T) \gamma(T) \Phi(T) \otimes T(\vector{f})(\vector{y}_0).
\end{equation*}

\end{proof}

\section{Application in constructing order conditions of the Runge--Kutta methods} \label{sec:order-condition-RK}

The framework for calculating the derivatives of TTNs under specific constraints was discussed in former sections. As an application, we use this framework in constructing order conditions of the RK methods.

For the ODEs system \eqref{eqn:ode-system}, a single step of an $s$-stage RK method is
\begin{equation}\left\{
\begin{aligned}
\label{RKscheme}
    \vector{Y}_i &= \vector{y}_0 + h \sum_{j=1}^{s} a_{ij} \vector{f}(\vector{Y}_j), \quad i = 1,\, \dots,\, s, \\
    \vector{y}_1 &= \vector{y}_0 + h \sum_{i=1}^{s} b_i \vector{f}(\vector{Y}_i).
\end{aligned}\right.
\end{equation}
Let us rewrite the above RK scheme into a matrix form
\begin{equation*}
    \begin{bmatrix} \vector{Y}_1 \\ \vdots \\ \vector{Y}_s \\ \vector{y}_1 \end{bmatrix}
    = \begin{bmatrix} \vector{y}_0 \\ \vdots \\ \vector{y}_0 \\ \vector{y}_0 \end{bmatrix} 
    + h \begin{bmatrix}
        a_{11} \matrix{I} & a_{12} \matrix{I} & \cdots & a_{1s} \matrix{I} & \matrix{0} \\
        \vdots   & \vdots   &        & \vdots   & \vdots \\
        a_{s1} \matrix{I} & a_{s2} \matrix{I} & \cdots & a_{ss} \matrix{I} & \matrix{0} \\
        b_1 \matrix{I}    & b_2 \matrix{I}    & \cdots & b_s \matrix{I}    & \matrix{0}
    \end{bmatrix} 
    \begin{bmatrix} \vector{f}(\vector{Y}_1) \\ \vdots \\ \vector{f}(\vector{Y}_s) \\ \vector{f}(\vector{y}_1) \end{bmatrix},
\end{equation*}
or a more compact natation
\begin{equation*}
    \Hat{\vector{Y}} = \Hat{\vector{Y}}_0 + h \Hat{\matrix{A}} \Hat{\vector{F}}(\Hat{\vector{Y}}),
\end{equation*}
where $\Hat{\vector{Y}} \in \mathbb{R}^{(s+1)d}$, $\Hat{\matrix{A}} = \begin{bmatrix}
    \Tilde{\matrix{A}} & \matrix{0} \\
    \vector{b}^\top \otimes \matrix{I}  & \matrix{0}
\end{bmatrix} = \begin{bmatrix}
    \matrix{A} \otimes \matrix{I} & \matrix{0} \\
    \vector{b}^\top \otimes \matrix{I}  & \matrix{0}
\end{bmatrix} \in \mathbb{R}^{(s+1)d \times (s+1)d}$, $\Hat{\vector{F}}(\Hat{\vector{Y}}) = \begin{bmatrix} \vector{F}(\vector{Y}) \\ \vector{f}(\vector{y}_1) \end{bmatrix} \in \mathbb{R}^{(s+1)d}$, and $\vector{F}(\vector{Y}) = \begin{bmatrix}
    \vector{f}(\vector{Y}_1) \\ \vdots \\ \vector{f}(\vector{Y}_s)
\end{bmatrix}$.

\subsection{Order conditions for RK methods}
Assume $\Hat{\Phi}(T)\in\mathbb{R}^{s+1}$ is a TTN that shares the same structure as $\Phi(T)$ and can be computed by replacing $\tilde{\matrix{A}} \matrix{I}_{s}^{l+1}$ with $\Hat{\matrix{A}} \tensor{I}_{s+1}^{l+1}$.
The order conditions for RK methods \eqref{RKscheme} are presented in the following theorem, which corresponds to \cite[Theorem 315A]{butcher2016numerical}. In contrast to the proof given in \cite[Theorem 315A]{butcher2016numerical}, our approach provides a constructive derivation of the order conditions, avoiding the use of mathematical induction.


\begin{theorem}
\label{theorem:ordercondition}
    In the case of the TTNs $T(\vector{f})$ are linearly independent, RK methods \eqref{RKscheme} have order $p$ if and only if 
    \begin{equation*}
    \gamma(T) \phi(T) = 1, \quad \forall |T| \leq p,
    \end{equation*}
    where $\phi(T) := (0,\,\dots,\, 0,\, 1)\Hat{\Phi}(T)$.
\end{theorem}
\begin{proof}
    By taking the limit $h \to 0$, we have 
\begin{equation*}
    \lim_{h \to 0} \Hat{\vector{F}}(\Hat{\vector{Y}}) = \begin{bmatrix}
        \vector{f}(\vector{y}_0) \\ \vdots \\ \vector{f}(\vector{y}_0) \\ \vector{f}(\vector{y}_0)
    \end{bmatrix}_{(s+1)d,1}
    = \vector{1}_{s+1} \otimes \vector{f}(\vector{y}_0).
\end{equation*}
Similarly, 
\begin{equation*}
    \lim_{h \to 0} \Hat{\vector{F}}^{(l)}(\Hat{\vector{Y}}) = 
    \matrix{I}_{s+1}^{l+1} \otimes \vector{f}^{(l)}(\vector{y}_0), \quad \text{for } l \geq 1.
\end{equation*}
Due to $\Hat{\matrix{A}} = \begin{bmatrix}
    \matrix{A} & \matrix{0} \\ \vector{b}^\top & 0
\end{bmatrix} \otimes \matrix{I}$, we have 
\begin{equation*}
    \lim_{h \to 0} \Hat{\matrix{A}} \Hat{\mathbf{\vector{F}}}(\Hat{\vector{Y}}) = \left( \begin{bmatrix}
    \matrix{A} & \matrix{0} \\ \vector{b}^\top & 0
\end{bmatrix} \otimes \matrix{I} \right) \left( \vector{1}_{s+1} \otimes \vector{f}|_{h=0} \right) 
= \left( \begin{bmatrix}
    \matrix{A} & \matrix{0} \\ \vector{b}^\top & 0
\end{bmatrix} \vector{1}_{s+1} \right) \otimes \vector{f}|_{h=0},
\end{equation*}
and
\begin{equation*}
    \lim_{h \to 0} \Hat{\matrix{A}} \Hat{\vector{F}}^{(l)}(\Hat{\vector{Y}}) = \left( \begin{bmatrix}
    \matrix{A} & \matrix{0} \\ \vector{b}^\top & 0
\end{bmatrix} \otimes \matrix{I} \right) \left( \matrix{I}_{s+1}^{l+1} \otimes \vector{f}^{(l)}|_{h=0} \right) 
= \left( \begin{bmatrix}
    \matrix{A} & \matrix{0} \\ \vector{b}^\top & 0
\end{bmatrix} \matrix{I}_{s+1}^{l+1} \right) \otimes \vector{f}^{(l)}|_{h=0}.
\end{equation*}
Therefore, the $k^\mathrm{th}$ order derivative of $\Hat{\vector{Y}}$ at $h = 0$ can be computed by
\begin{equation*}
    \Hat{\vector{Y}}^{(k)} \big|_{h=0} = \sum_{|T| = k} \alpha(T) \gamma(T) \Hat{\Phi}(T) \otimes T(\vector{f})|_{h=0},
\end{equation*}
where $\Hat{\Phi}(T) \in \mathbb{R}^{s+1}$ is independent of $\vector{f}$. 

With these results, the $k^\mathrm{th}$ order derivative of numerical solutions $\vector{y}_1$ is
\begin{equation}\label{eqn:RK-final-expansion}
\begin{aligned}
    \vector{y}_1^{(k)} \big|_{h=0} &= ((0,\dots,0,1) \otimes \matrix{I}) \Hat{\vector{Y}}^{(k)}|_{h=0} \\
    &= \sum_{|T| = k} \alpha(T) \gamma(T) \left( (0,\,\dots,\, 0,\, 1)\Hat{\Phi}(T) \right) \otimes T(\vector{f})|_{h=0} \\
    &= \sum_{|T| = k} \alpha(T) \gamma(T) \phi(T)\, T(\vector{f})|_{h=0},
\end{aligned}
\end{equation}
where $\phi(T)$ is a scalar defined as $\phi(T) := (0,\,\dots,\, 0,\, 1)\Hat{\Phi}(T)$. By comparing \eqref{eqn:RK-final-expansion} with \eqref{eqn:ode-expansion-a}, a sufficient condition for an $s$-stage RK scheme to be of order $p$ is immediately obtained:
\begin{equation*}
    \gamma(T) \phi(T) = 1, \quad \forall |T| \leq p.
\end{equation*}
It is also a necessary condition, since the TTNs $T(\vector{f})$ are linearly independent, which is consistent with the independence of elementary differentials \cite[Sec. II.2, page 155]{hairer1987solving}. 
\end{proof}


\subsection{Super convergence of RK methods for a specific function {\it\textbf f}}
\label{subsec:Super convergence of RK methods for a specific function}

In \Cref{theorem:ordercondition}, the proof of the necessary condition relies on the linear independence of the TTNs $T(\vector{f})$. While the linear independence of Butcher trees $T$ (elementary differentials) is established in \cite[Theorem 314A]{butcher2016numerical}, the situation differs when considering TTNs $T(\vector{f})$ for a fixed function $\vector{f}$. Specifically, two distinct trees $T_1\neq T_2$ may produce linearly dependent TTNs, i.e., $T_1(\vector{f})$ and $T_2(\vector{f})$, under certain conditions on $\vector{f}$. For example, if $\vector{f}''(\vector{y})=0$ and both $T_1(\vector{f})$ and $T_2(\vector{f})$ involve the second derivative $\vector{f}''(\vector{y})$, then we have $T_1(\vector{f})=T_2(\vector{f})=0$, indicating linear dependence.

Let us define the linear space $\mathbb{T}_p(\vector{f})$ as: 
\begin{equation}
    \mathbb{T}_p(\vector{f}):=\hbox{span}\left\{ T(\vector{f}),\,
    \forall |T|= p\right\}.
\end{equation}
Assume that $\dim \mathbb{T}_p(\vector{f})=n_p\leq N_p$, where $N_p$ represents the number of all distinct unlabelled trees of order $p$. Let $T_1,$ $\ldots$, $T_{N_p}$ denote all distinct trees of order $p$, and let ${\mathcal T}_{p,\,1}$, $\ldots$, ${\mathcal T}_{p,\,n_p}$ be a basis of $\mathbb{T}_p(\vector{f})$. Then we can write 
\begin{equation}
    (T_1(\vector{f}),\,\ldots,\,T_{N_p}(\vector{f}))= ({\mathcal T}_{p,\,1},\,\ldots,\,{\mathcal T}_{p,\,n_p}){\mathcal M}_p,
\end{equation}
where ${\mathcal M}_p\in\mathbb{R}_{n_p\times N_p}$ is a real matrix of rank $n_p$. Define the row vector:
\begin{equation}
    \vector{\alpha}_p=\left(\alpha(T_1),\,\ldots,\,\alpha(T_{N_p})\right)\in\mathbb{R}^{N_p},
\end{equation}
and the diagonal matrix
\begin{equation}
    \tensor{W}_p=\hbox{diag}\left\{\gamma(T_1)\phi(T_1),\,\ldots,\,\gamma(T_{N_p})\phi(T_{N_p})\right\}\in\mathbb{R}^{{N_p}\times{N_p}}.
\end{equation}

With these notations, we present both the sufficient and necessary conditions for RK methods with a given $\vector{f}$, as stated in the following theorem.

\begin{theorem}
\label{theorem:ordercondition-a}
    RK methods \eqref{RKscheme} have order $p$ if and only if 
    \begin{equation*}
    \vector{\alpha}_j {\mathcal M}^\top_j = \vector{\alpha}_j \tensor{W}_j{\mathcal M}^\top_j, \quad \forall j \leq p.
    \end{equation*}
\end{theorem}
\begin{proof}
According to \Cref{theorem:taylorexpansionfory}, the Taylor expansion of $\vector{y}(t)$ satisfying the ODEs system \eqref{eqn:ode-system} at $t = 0$ is given by:
\begin{equation}
\label{theorem12:eq:1}
\begin{aligned}
    \vector{y}(t)& = \vector{y}(0) + \sum_{k=1}^{\infty} \frac{1}{k!} t^k \sum_{|T|=k} \alpha(T) \left(T(\vector{f})\right)(\vector{y}(0)),\\
    &=\vector{y}(0) + \sum_{k=1}^{\infty} \frac{1}{k!} t^k\left<\vector{\alpha}_k{\mathcal M}^\top_k,  ({\mathcal T}_{k,\,1},\,\ldots,\,{\mathcal T}_{k,\,n_k})\right>_k,
\end{aligned}
\end{equation}
where $\left<\cdot,  \cdot\right>_k$ denotes the standard inner product in $\mathbb{R}^{n_k}$.

On the other hand, from \eqref{eqn:RK-final-expansion}, the Taylor expansion of numerical solution $\vector{y}_1$ generated by the RK scheme \eqref{RKscheme} at $h = 0$ is:
\begin{equation}
\label{theorem12:eq:2}
\begin{aligned}
    \vector{y}_1(h)& = \vector{y}_0 + \sum_{k=1}^{\infty} \frac{1}{k!} h^k \sum_{|T| = k} \alpha(T) \gamma(T) \phi(T)\, T(\vector{f})|_{h=0},\\
    &=\vector{y}_0 + \sum_{k=1}^{\infty} \frac{1}{k!} h^k\left<\vector{\alpha}_k\tensor{W}_k {\mathcal M}^\top_k,  ({\mathcal T}_{k,\,1},\,\ldots,\,{\mathcal T}_{k,\,n_k})\right>_k.
\end{aligned}
\end{equation}

By comparing \eqref{theorem12:eq:1} and \eqref{theorem12:eq:2}, and using the linear independence of ${\mathcal T}_{k,\,j}$, we conclude that the RK method \eqref{RKscheme} achieves order $p$ if and only if 
    \begin{equation*}
    \vector{\alpha}_j {\mathcal M}^\top_j = \vector{\alpha}_j \tensor{W}_j{\mathcal M}^\top_j, \quad \forall j \leq p.
    \end{equation*}
\end{proof}

It should be noted that the choice of basis does not alter the order conditions stated in \Cref{theorem:ordercondition-a}. Assume that three is another basis ${\mathcal T}'_{j,\,1}$, $\ldots$, ${\mathcal T}'_{j,\,n_j}$ such that
\begin{equation}
    (T_1(\vector{f}),\,\ldots,\,T_{N_j}(\vector{f}))= ({\mathcal T}'_{j,\,1},\,\ldots,\,{\mathcal T}'_{j,\,n_j}){\mathcal M}'_j.
\end{equation}
Then the order conditions derived from \Cref{theorem:ordercondition-a} are:
    \begin{equation}
    \label{orderconditionforanotherbasis}
    \vector{\alpha}_j ({\mathcal M}'_j)^\top = \vector{\alpha}_j \tensor{W}_j({\mathcal M}'_j)^\top, \quad \forall j \leq p.
    \end{equation}
Let matrix $\tensor{G}_j\in\mathbb{R}^{n_j\times n_j}$ be the transition matrix from ${\mathcal T}'_{j,\,1}$, $\ldots$, ${\mathcal T}'_{j,\,n_j}$ to ${\mathcal T}_{j,\,1}$, $\ldots$, ${\mathcal T}_{j,\,n_j}$. Then we have 
\begin{equation}
     ({\mathcal T}'_{j,\,1},\,\ldots,\,{\mathcal T}'_{j,\,n_j}){\mathcal M}'_j=({\mathcal T}_{j,\,1},\,\ldots,\,{\mathcal T}_{j,\,n_j})\tensor{G}_j{\mathcal M}'_j=({\mathcal T}_{j,\,1},\,\ldots,\,{\mathcal T}_{j,\,n_j}){\mathcal M}_j,
\end{equation}
which implies $\tensor{G}_j{\mathcal M}'_j={\mathcal M}_j$. Since the transition matrix $\tensor{G}_j$ is invertible, multiplying both sides of \eqref{orderconditionforanotherbasis} by $\tensor{G}_j^\top$, we obtain:
    \begin{equation}
    \vector{\alpha}_j {\mathcal M}_j^\top = \vector{\alpha}_j \tensor{W}_j{\mathcal M}_j^\top, \quad \forall j \leq p,
    \end{equation}
    which shows that the order conditions stated in \Cref{theorem:ordercondition-a} are invariant under the choice of basis.

For $N_k$ distinct unlabelled trees of order $k$, if the corresponding TTNs $T(\vector{f})$ are linearly independent, they can be chosen as the basis of $\mathbb{T}_k$. In this case, we have ${\mathcal M}_k=I_{N_k}$ and \Cref{theorem:ordercondition-a} becomes equivalent to \Cref{theorem:ordercondition}. 
In general, \Cref{theorem:ordercondition-a} potentially imposes fewer restrictions than \Cref{theorem:ordercondition}. This suggests that an RK scheme may attain the order $p$ for arbitrary functions $\vector{f}$, yet exhibit a higher (superior) convergence order for certain specific choices of $\vector{f}$. We illustrate the existence of such super convergence through an example, which is also discussed in \cite{butcher2016numerical}.

\begin{example}
    Consider a scalar problem 
    \begin{equation}\label{eqn:scalar-ode}
        y'(t) = f(y(t), t),
    \end{equation}
    where both $y(t)$ and $f(y,t)$ are scalar-valued functions. The problem \eqref{eqn:scalar-ode} can be reformulated in autonomous form as: 
    \begin{equation}\label{eqn:scalar-ode-autonomous}
    \frac{\dd }{\dd t}
        \begin{bmatrix}
            y(t) \\ t
        \end{bmatrix}
        = \begin{bmatrix}
            f(y(t), t) \\
            1
        \end{bmatrix}.
    \end{equation}
    Let us denote $\vector{Y} := \begin{bmatrix} y(t) \\ t \end{bmatrix}$ and $\vector{F}(\vector{Y}) := \begin{bmatrix} f(y(t), t) \\ 1 \end{bmatrix}$. According to \Cref{theorem:ordercondition-a}, the bases of the linear spaces $\mathbb{T}_j(\vector{F})$, $j=1,\,\ldots,\,p$, determine the order conditions for applying RK methods to this equation. While the TTNs $T(\vector{F})$ are linearly independent for $p \leq 4$, linear dependencies begin to appear when $p \geq 5$. For example, when $p = 5$, consider the trees $T_1 = \begin{forest} [ [[[]]] [] ] \end{forest}$, and $T_2 = \begin{forest} [ [ [[]][] ] ] \end{forest}$. The corresponding TTNs $ T_1(\vector{F})$ and $ T_2(\vector{F})$  are identical and can be written as:
    \begin{equation}
       T_1(\vector{F})= T_2(\vector{F}):=\begin{bmatrix}
            f_y (f_{yy} f + f_{ty}) (f_y f + f_t) \\ 0
        \end{bmatrix}.
    \end{equation}

    If we choose $f(y(t),t)$ such that  $\dim \mathbb{T}_5(\vector{F})=N_5-1$, then by \Cref{theorem:ordercondition-a}, the order conditions for applying RK methods to ODEs system \eqref{eqn:scalar-ode-autonomous} are:
    \begin{equation*}
        \gamma(T) \phi(T) = 1,\, \forall |T|\leq 5,\,\,\,\hbox{ and  } \,T\neq T_1,\, T\neq T_2,  
    \end{equation*}
    and
    \begin{equation*}
        \alpha(T_1) \gamma(T_1) \phi(T_1) + \alpha(T_2) \gamma(T_2) \phi(T_2) = \alpha(T_1) + \alpha(T_2). 
    \end{equation*}
    This order conditions differ from those required for general vector-valued $\vector{f}$ when $p = 5$, which are given by \Cref{theorem:ordercondition} as:
    \begin{equation*}
        \gamma(T) \phi(T) = 1,\,\,\, \forall |T|\leq 5.   
    \end{equation*}
    
    This property can lead to super convergence of RK methods. In \cite[page 176]{butcher2016numerical}, Butcher present the following RK method which has classical order 4 for general vector-valued functions $\vector{F}$, achieves order 5 when applied to the scalar problem \eqref{eqn:scalar-ode}. The tableau of the given RK method is: 
    \begin{equation*}
    \def\arraystretch{1.5}
    \setlength{\arraycolsep}{10pt}
    \begin{array}{c|cccccc}
    0 & & & & & &  \\
    \frac{1}{2} & \frac{1}{2} & & & & &  \\
    1 & -\frac{9}{4} & \frac{13}{4} & & & &  \\
    \frac{1}{4} & \frac{9}{64} & \frac{5}{32} & -\frac{3}{64} & & &  \\
    \frac{7}{10} & \frac{63}{625} & \frac{259}{2500} & \frac{231}{2500} & \frac{252}{625} & &  \\
    1 & -\frac{27}{50} & -\frac{139}{50} & -\frac{21}{50} & \frac{56}{25} & \frac{5}{2} &  \\
    \hline
    & \frac{1}{14} & 0 & 0 & \frac{32}{81} & \frac{250}{567} & \frac{5}{54}  \\
\end{array}.
\end{equation*}

For the trees $T_1$ and $T_2$, we have
    \begin{equation*}
    \label{rkscheme4order}
    \begin{aligned}
        \alpha(T_1) & =4,\quad\gamma(T_1)=30,\quad \hbox{and}\quad\phi(T_1) = \frac 1 {30}+\frac 3{320},\\
        \alpha(T_2) &=3,\quad\gamma(T_2)=40,\quad \hbox{and}\quad\phi(T_2) = \frac 1 {40}-\frac 3{320},
    \end{aligned}
    \end{equation*}
    which implies that $\gamma(T)\phi(T)=1$ does not hold for $T=T_1$ or $T=T_2$. Therefore, the RK scheme described above is of order four for general two-dimensional vector functions $\vector{F}$, such as:
        \begin{equation}
    \frac{\dd }{\dd t}
        \begin{bmatrix}
            x \\ y
        \end{bmatrix}
        = \begin{bmatrix}
            \frac{x+y}{\sqrt{x^2+y^2}} \\
            \frac{x-y}{\sqrt{x^2+y^2}}
        \end{bmatrix}.
    \end{equation}
    However, for the scalar problem \eqref{eqn:scalar-ode}, the following order condition holds:
        \begin{equation*}
        \begin{aligned}
        \alpha(T_1) \gamma(T_1) \phi(T_1) &+ \alpha(T_2) \gamma(T_2) \phi(T_2) \\
        &= 4 \times 30 \times \left(\frac 1 {30}+\frac 3{320}\right ) +3 \times 40 \times \left(\frac 1 {40}-\frac 3{320}\right )\\&=7=\alpha(T_1) + \alpha(T_2), 
        \end{aligned}
    \end{equation*}
    which indicates that the RK scheme \eqref{rkscheme4order} is of fifth order for the scalar problem.

\end{example}

\subsection{Comparison of our framework and Butcher's method.}
\label{subsec:Comparison of our framework and Butcher's method}
We enclose this section with a summary of the similarities and differences between our approach and Butcher's method \cite{butcher1963coefficients, butcher2016numerical}.

\begin{enumerate}
    \item Both methods utilize the summation of trees to represent the high order derivatives of a vector valued function $\vector{y}$ under certain constraints. A key distinction between our approach and Butcher’s method lies in the treatment of trees. Butcher's method introduces trees through graphs and discusses operations, such as Butcher product, on trees within a ``forest'', where the trees are more akin to abstract algebraic symbols. In our framework, the trees, referred to as TTNs, are formulated purely in terms of tensor operations and tensor derivatives, which are treated as tree tensors. As a result, we can compute the products, derivatives, contractions of trees following the tensor framework. This makes the operations on tree easier to understand and more intuitive. 

    \item Both methods introduce decompositions for trees to derive a recursive approach for computing high order derivatives and the multiplication of high order trees. This decomposition can be derived from the defined growth process of the tree. Butcher's method recursively decomposes a tree into several subtrees. This decomposition is straightforward to follow and understand when treating the tree as a graph, but its connection to derivatives acting on the tree is less clear. To establish a more clear connection between the derivatives and the growth processes of trees, we present two perspectives to understand tree growth. The first perspective is given in \autoref{sec:3.1}, which uses \eqref{eqn:high-order-derivative-path-1} to grow a tree from the root, adding one leaf at a time. This one-leaf growth process corresponds to taking one order of derivative and then immediately substituting the constraint into the new leaf. Based on \eqref{eqn:high-order-derivative-y-general-expression}, the second perspective is provided in \autoref{subsec:Differentiation following layer-wise growth of trees}, where the tree grows layer by layer from the root. This one-layer growth process corresponds to taking derivatives without constrains as much as possible and then substituting the constrains into the leaves. These two approaches offer a clear and systematic way to understand differentiation under algebraic or ODE constraints.   

    \item Since the decomposition of high order trees differs between the two methods, the subsequent strategies for computing the multiplication $\alpha(T)$ of tree $T$ also diverge. In the first approach, we take one order of derivative and then immediately substitute the constraint into the new leaf. This introduces a new concept, namely the differentiation path. We study the differentiation path using the valid labelled trees. Consequently, the multiplication $\alpha(T)$ is equal to the number of valid labelled trees that can be derived from an unlabelled tree $T$. In the second approach of our framework, as discussed in \autoref{subsec:Differentiation following layer-wise growth of trees}, we first take derivatives without constraints as much as possible. The contribution of this process to the multiplication $\alpha(T)$ can be easily counted using the method of undetermined coefficients. Unlike Butcher's approach, these two approaches offer new insights into how differentiation operations contribute to the multiplicity $\alpha(T)$, enriching the theoretical understanding from a fresh perspective. 

    \item In the construction of order conditions for the RK methods, Butcher’s classical approach relies on operations involving rooted trees, while our method utilizes the algebraic framework of TTNs. This tensor-based formulation provides a systematic and structured way to handle derivatives arising from the RK methods, offering the following significant advantages compared with Butcher's method. First, tensors are algebraic entities so that plenty of properties of tensors and operations on tensors can be utilized. In our framework,  the contracted product and the matrix-vector description of the RK method are utilized to provide a concise Taylor expansion of the numerical solution for the RK methods.
Using the Kronecker product and mixed product property of tensors, we decompose the TTNs in Taylor expansion of the numerical solution into two parts. The first part  depends solely on the coefficients matrix $A$ in RK method, while the second part is identical to $T(\vector{f})$. 
With this decomposition, the Taylor expansion is continuously simplified in a summation similar to the Taylor expansions of the exact solution, except for the weights of $T(\vector{f})$. Therefore, the second advantage of our method is that the proof of order conditions is achieved by directly comparing the Taylor expansion of both the exact and numerical solutions, eliminating the need for mathematical induction.
Third, tensor operations provide a structured approach for representing and manipulating derivatives, making it easier to generalize to other numerical methods beyond the standard RK framework. 

\item Both of our framework and Butcher's method yield the same uniform order conditions for RK schemes, meaning that the conditions hold for any vector-valued function $\vector{f}(\vector{y})$, as stated in \Cref{theorem:ordercondition}. However, for certain specific choices of $\vector{f}$, the order conditions proposed in \Cref{theorem:ordercondition} are not strictly necessary, as demonstrated by an example in \autoref{subsec:Super convergence of RK methods for a specific function}. This is due to inherent linear dependencies among the TTNs $T(\vector{f})$. To address this, we remove such redundancies by constructing a basis for the corresponding linear space, leading to sharper order conditions tailored to a given $\vector{f}$. By verifying these refined conditions, we can determine whether a standard RK scheme of order $p$ actually exhibits a higher (superior) convergence order for a specific $\vector{f}$.  
\end{enumerate}


In summary, based on the tensor operators and the clear connection between TTNs and derivatives, our framework can be naturally extended to other RK-type schemes, such as additive RK methods \cite{kennedy2003additive} (e.g. the well-known Implicit-Explicit methods \cite{ascher1997implicit}),  partitioned Runge--Kutta \cite{hairer1981order}, and nonlinearly partitioned Runge--Kutta methods \cite{buvoli2024new, tran2024order}. Similar to the standard RK method, we can also approximate the numerical solutions of these RK methods using a Tayloy expansion, which is a summation of TTNs $T(\vector{f})$ with different weights depending on the coefficient matrix of the corresponding RK method. 
This flexibility highlights the broad applicability of our approach, enabling a unified analysis of various RK-type methods.

\section{Conclusion}

In this work, we have developed an innovative mathematical framework for computing partial and total derivatives of functional TTNs, establishing new theoretical foundations with broad applicability in computational mathematics. Through rigorous analysis of constrained vector-valued functions as representative examples, we have demonstrated both the effectiveness and mathematical rigor of our framework in computing high-order derivatives and Taylor expansions. The tensor-algebraic formulation of these expansions provides a powerful tool that significantly simplifies the derivation of order conditions for RK method. Importantly, our framework admits natural extensions to advanced RK variants, including additive RK, partitioned RK, and nonlinear partitioned RK methods, offering a unified approach to order condition analysis across these related numerical schemes. 
In summary, this TTN-based derivatives offers profound theoretical insights into the mathematical structure of tensor networks while simultaneously providing practical tools for addressing real-world scientific computing challenges.


\bibliographystyle{amsplain}

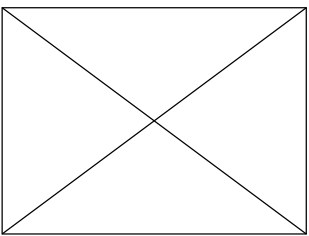
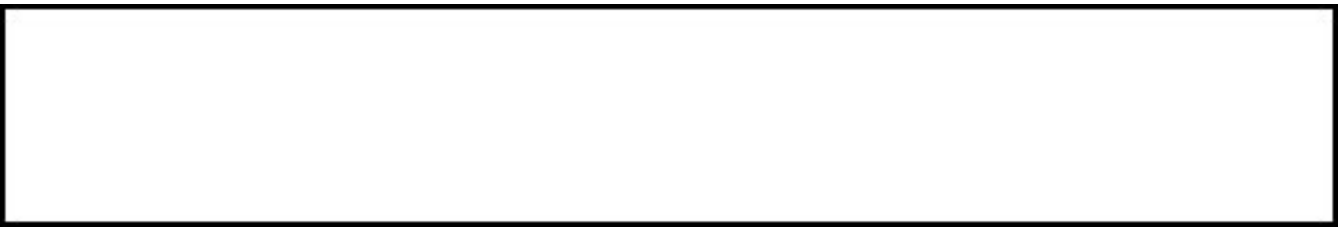
\end{document}